\documentclass[11 pt]{amsart}
\usepackage[usenames,dvipsnames]{x
color}
\usepackage{amsmath, amsthm,amssymb,bbm,enumerate,mathrsfs,mathtools}
\usepackage{a4wide}
\usepackage{tikz,verbatim, graphicx}
\usetikzlibrary{calc,shapes}
\usepackage{hyperref}
\usepackage{makecell}
\usepackage{verbatim}

\renewcommand\labelenumi{(\roman{enumi})}
\renewcommand\theenumi\labelenumi

\tikzset{
  bigblack/.style={circle, draw=black!100,fill=black!100,thick, inner sep=1.5pt, minimum size=6mm},
  medblack/.style={circle, draw=black!100,fill=black!100,thick, inner sep=1.5pt, minimum size=4mm},
  blackcirc/.style={circle, draw=black!100,thick, inner sep=1.5pt, minimum size=6mm},  
}

\title{\vspace{-0.8cm}The K\H{o}nig Graph Process}
\author{ Nina Kam\v{c}ev 
\and Michael Krivelevich 
\and Natasha Morrison 
\and Benny Sudakov }

\address{School of Mathematics, Monash University, VIC 3800, Australia}
\email{nina.kamcev@monash.edu}
\address{Department of Mathematics, ETH, Zurich, Switzerland}
\email{benjamin.sudakov@math.ethz.ch}
\address{School of Mathematical Sciences, Tel Aviv, Israel}
\email{krivelev@tauex.tau.ac.il}
\address{Department of Pure Mathematics and Mathematical Statistics, University of Cambridge, Wilberforce Road, Cambridge, UK}
\email{morrison@dpmms.cam.ac.uk}
\thanks{The second author is partially supported by USA-Israel BSF grants 2014361 and 2018267, and by ISF grant 1261/17.}
\thanks{The third author is supported by a research fellowship from Sidney Sussex College, Cambridge.}
\thanks{The fourth author is supported by an SNSF grant 200021-17557.}

\newtheoremstyle{case}{}{}{\normalfont}{}{\itshape}{:}{ }{}
 
\newtheorem{thm}{Theorem}
\newtheorem{lem}[thm]{Lemma}
\newtheorem{prop}[thm]{Proposition}

\newtheorem{cor}[thm]{Corollary}
\newtheorem{claim}[thm]{Claim}
\newtheorem{ques}[thm]{Question}

\newtheorem*{thmi}{Theorem~\ref{thm:finalgraph}\eqref{itm:pm}}
\newtheorem*{thmii}{Theorem~\ref{thm:finalgraph}\eqref{itm:ovc}}

\theoremstyle{definition}

\newtheoremstyle{case}{}{}{\normalfont}{}{\itshape}{\normalfont:}{ }{}

\theoremstyle{case}

\numberwithin{equation}{section}
\numberwithin{thm}{section}

\newcommand\given[1][]{\:#1  \vert  \:}
\newcommand{\Bin}{\textrm{Bin}}

\newcommand{\Var}[1]{\textrm{Var} \left[ #1 \right]}
\newcommand{\pr}[1]{\mathbb{P} \left[ #1 \right]}
\newcommand{\prm}[2]{\mathbb{P}_{#1} \left[ #2 \right]}
\newcommand{\er}[1]{\mathbb{E} \left[ #1 \right]}
\newcommand{\erm}[2]{\mathbb{E}_{#1} \left[ #2 \right]}

\newcommand{\N}{\mathbb{N}}
\newtheorem*{pro2}{The Random K\H{o}nig Hypergraph Process}

\newcommand{\dist}{\text{dist}}
\newcommand{\Gall}[1]{G^{\text{all}}_{#1}}
\newcommand{\Gour}[1]{G_{#1}}
\newcommand{\Gnm}{G_{n, m}}
\newcommand{\Gnp}{G(n,p)}

\newcommand{\cfinal}{C_{N}}

\newcommand{\phie}{\varphi}

\newcommand{\ck}{\hat{I}}
    \newcommand{\three}{a}

\newcommand{\Ont}{O \left( n^{-2}\right)}

\def\nk#1{[\emph{#1}]}

\begin{document}

\begin{abstract}
Say that a graph $G$ has property $\mathcal{K}$ if the size of its maximum matching is equal to the order of a minimal vertex cover. We study the following process. Set $N:= \binom{n}{2}$ and let $e_1, e_2, \dots e_{N}$ be a uniformly random ordering of the edges of $K_n$, with $n$ an even integer. Let $\Gour{0}$ be the empty graph on $n$ vertices. For $m \geq 0$, $\Gour{m+1}$ is obtained from $\Gour{m}$ by adding the edge $e_{m+1}$ exactly if $\Gour{m} \cup \{ e_{m+1}\}$ has property $\mathcal{K}$. We analyse the behaviour of this process, focusing mainly on two questions: What can be said about the structure of $\Gour{N}$ and for which $m$ will $\Gour{m}$ contain a perfect matching? 
\end{abstract}

\maketitle


\section{Introduction}
	
The modern study of random graph processes began in 1959 with the inaugural papers of Erd\H os and R\'enyi~\cite{er59,er60}. Given a uniformly random permutation $e_1,\ldots, e_N$ of $E(K_n)$, they studied the evolution and properties of the graph $\Gnm$ with edge set $\{e_1,\ldots,e_m\}$, which is now known as the \emph{Erd\H os-R\'enyi random graph}. This work has since grown into a well-established research area with many important applications in theoretical computer science, statistical physics, and other branches of mathematics \cite{bollobas01,jlr,fk16}.

An important variant of the standard Erd\H os-R\'enyi process, often referred to as the \emph{random greedy process}, is the following. Given a graph property $\mathcal{P}$, preserved by the removal of edges, begin with an empty $n$-vertex graph and at each step add an edge chosen uniformly at random from those that do not violate property $\mathcal{P}$. The random greedy process was first considered by Ruci\'{n}ski and Wormald~\cite{rw92} (in the case of bounded degree) and, following discussions of Bollob\'{a}s and Erd\H os, by Erd\H{o}s, Suen and Winkler in 1995~\cite{esw95} (in the case of triangle-freeness). Their motivation was defining and analysing a natural probability measure on the set of $\mathcal{P}$-maximal graphs.

A particularly well studied property is that of being $H$-free for a general graph $H$. In many cases, the final graph obtained at the end of the $H$-free process has been used to give constructions of interest in extremal combinatorics. In particular, such constructions have been found to improve lower bounds on Tur\'an numbers (see \cite{wolfovitz09,bk10}) and on off-diagonal Ramsey numbers (for example \cite{bohman09,bk10,bk13,fgs17}).

In addition to looking at the structure and properties of the final graph, one often asks questions about the evolution of the process itself (see, e.g., \cite{bk10,ksv09}).  The properties mentioned so far are \emph{decreasing} (closed under removal of edges) and \emph{local}. Monotonicity of $\mathcal{P}$ guarantees that the final graph $G_N$ is maximal in $\mathcal{P}$ and facilitates the use of some common techniques such as coupling with a modified process. So far, global properties are far less well understood and there is no standard approach to analysing these processes (see for instance, the properties of being planar \cite{gsst08}, $r$-colourable \cite{ksv09} and $k$-matching-free \cite{kkls17}).

In this paper we consider a global non-monotone property of a graph $G$, that the size of a maximum matching $\nu(G)$ is equal to the cardinality of an optimal vertex cover $\tau(G)$. A \emph{vertex cover} in $G$ is a set of vertices incident to any edge of $G$. Equivalently, the complement of a vertex cover contains no edges of $G$ and is thus an \emph{independent set}. We say that the vertex cover $C$ is \emph{optimal} if there is no vertex cover of cardinality less than $C$. It is easy to see that in general $ \nu(G) \leq \tau(G) \leq 2\nu(G)$. We say that $G$ is a \emph{K\H{o}nig graph} (or has property $\mathcal{K}$) 
if $\nu(G) = \tau(G)$.

The properties $\nu(G)$ and $\tau(G)$ and the relationship between them have been studied in many contexts. A foundational theorem of K\H{o}nig and independently Egerv\'ary \cite{konig31,egervary31} says that bipartite graphs have the K\H{o}nig property. 
The problem of finding an optimal vertex cover NP-hard but it can be solved in polynomial time in K\H{o}nig graphs via the maximum matching. However, {most} graphs are closer to the other end of the spectrum, where $\tau(G) \sim 2 \nu(G)$. As, with high probability, $\Gnm$ with $m = \frac 12n (\log n  + \omega(1))$ has a perfect matching~\cite{er66}, whereas $\tau(\Gnm) \sim n$ for $m \gg n$ \cite{jlr}.
\footnote{As usual, we say an event occurs \emph{with high probability} if it occurs with probability tending to 1 as $n \rightarrow \infty$.\\
We write    $f\sim g$ if $f/g\to 1$,  $f=o(g)$ or $f \ll g$ if $f/g\to 0$,    $f=O(g)$ if $|f|\le Cg$ for some constant $C$.  We write $\omega(1)$ to denote a function tending to infinity with underlying parameter n; we also write $\omega(n)$ for $\omega(n)=\omega(1)\cdot n$.}


In light of this, we are interested in the evolution of a random graph process constrained by the K\H{o}nig property, defined as follows. Let $\Gour{0}$ be the empty graph on vertex set $V$, where $|V| = n$ and $N= \binom{n}{2}$. Let $e_1, e_2, \dots e_{N}$ be a uniformly random ordering of the edges of the complete graph on $K_n$ on $V$. At each step $m \geq 0$, the edge $e_{m+1}$ is \emph{offered} to $\Gour{m}$. Say that a vertex pair $f$ is \emph{acceptable} for $G_m$ if $f \notin  \{e_1,\ldots,e_m\}$ and $G_m + f$ has property $\mathcal{K}$.  If $e_{m+1}$ is acceptable for $\Gour{m}$, we set $\Gour{m+1}:= \Gour{m} + e_{m+1}$ and say that the edge $e_{m+1}$ is \emph{accepted}. Otherwise we say that $e_{m+1}$ is \emph{rejected} and set $\Gour{m+1} := \Gour{m}$. An important observation is that an edge $e$ is acceptable for $\Gour{m}$ if and only if $e$ is incident to an optimal cover in $\Gour{m}$ or extends a maximum matching in $\Gour{m}$. In the remainder of the paper, we assume that the number of vertices $n$ is even. Let $\Gall{m}$ be the graph whose edge set is $\{e_1, \dots, e_m\}$. 
Note that $\Gall{m}$ is distributed as $\Gnm$. All the proofs translate to odd $n$ if we define a perfect matching to be a matching of order $\left \lfloor \frac n2 \right \rfloor$.

 We remark that it is also natural to consider an alternative process in which   $\tau(G)=\nu(G)$ is maintained not just for $G_m$, but for every subgraph of it. However, this condition is equivalent to bipartiteness and yields precisely the bipartite graph process considered by Erd\H{o}s, Suen and  Winkler~\cite{esw95}. Our process and the resulting graph are rather different. 

Let us return to the K\H{o}nig process and consider what can be said about the structure of $G_N$. We start with a simple proposition which is proved in Section~\ref{section:matching}.

\begin{prop} \label{prop:deterministic}
	For all $m \ge 1$, the graph $\Gour{m}$ has a maximum matching that intersects every edge of $\Gall{m}$. 
\end{prop}

At the end of the process, since all the edges of $K_n$ have been offered, deterministically the final graph $\Gour{N}$ has a perfect matching. 
This settles the value of $\nu(G_N)$ and raises a number of further questions about the typical structure of $G_N$ and the evolution of the process (in particular, appearance of a perfect matching). As $G_N$ has a vertex cover $C$ of order $\frac n2$, it is a subgraph of the classical Erd\H{o}s-Gallai graph $G^*_n$, which consists of a clique on $\frac n2$ vertices that is completely joined to an independent set on the other $\frac n2$ vertices. 
Given this, we wonder how close is the graph $\Gour{N}$ to $G^*_n$, or in other words, how many vertex pairs incident to $C$ are missing? How `volatile' is the optimal cover typically  in the initial stages of the process, and at which point does it become `rigid' or unique? This question is also important for understanding the evolution of the process, in particular the proportion of acceptable edges. Our first main result is the following.

\begin{thm}
    \label{thm:finalgraph}
    Let $\epsilon > 0$. With high probability, the K\H{o}nig process satisfies the following properties.
    \begin{enumerate}[(i)] 
    \item\label{itm:pm} $\Gour{(1 + \epsilon)n \log n}$ contains a perfect matching.
    \item\label{itm:ovc} $\Gour{4n \log n}$ has a unique optimal vertex cover $C$. 
    \item\label{itm:me} There are $O(n)$ vertex pairs incident to $C$ that are not present in $\Gour{N}$.
    \end{enumerate}
\end{thm}

 Furthermore, we analyse in more detail the appearance of a perfect matching in $\Gour{m}$. In the Erd\H{o}s-R\'{e}nyi process $(\Gnm)$, Boll\'obas and Thomason \cite{bt83} proved that the very edge that links the last isolated vertex to another vertex makes the graph connected and completes a perfect matching with high probability if $n$ is even. This happens when $m=\frac 12 n(\log n + \omega(1))$.  In fact, they showed that if $m \geq  \frac 14 n \left( \log n + \omega(1) \right)$, then $\Gnm$ contains a perfect matching on all but at most one non-isolated vertex with high probability. For $m < \frac 14 n \log n$, there are other structural obstructions to containing a perfect matching -- in this regime, $\Gnm $ is likely to contain vertices with two neighbours of degree one, only one of which can be contained in a perfect matching. However, \emph{quantitatively}, the isolated vertices are the main obstruction throughout the evolution of the process, as shown by Frieze \cite{frieze86}. 

Our second main result is an analogue of \cite{bt83} for our process. We show that, with high probability, a perfect matching in  $\Gour{m}$ occurs significantly later than in $\Gnm$.

\begin{thm}
\label{thm:delayedpm}
 Let $m:= \left(\frac{1}{2} + \frac{1}{70} \right) n \log n$. With high probability, $\Gour{m}$ contains isolated vertices.
\end{thm}
This delay in $\Gour{m}$ is surprising, as one might guess that the number of isolated vertices decays roughly at the same rate as in $\Gnm$. This would indeed be the case if most pairs containing isolated vertices in $\Gour{m}$ were to extend the current maximum matching. 

The delay is, in spirit, similar to \emph{Achlioptas processes}, which were conceived for the sake of influencing the typical appearance of graph theoretic properties.   Initiating a fruitful line of research, Bohman and Frieze exhibited an Achlioptas process with a delayed phase transition~\cite{bf01}. Besides the phase transition, several other graph properties were considered. For instance, it is known that connectivity and occurrence of a Hamilton cycle (and hence a perfect matching) can  be accelerated \cite{mp14,kls10}. 

As property $\mathcal{K}$ is non-monotone and global, many of our arguments involve novel ideas (to our knowledge) and could be of their own interest or be adapted to study other global 
properties.  Our approach combines probabilistic arguments with combinatorial maximum-matching methods, resulting in intuitive and conceptually simple proofs. 

The paper is structured as follows. In the next section, we introduce standard notation that will be used throughout and outline in detail how the proof will proceed. Then in Section~\ref{section:prelim} we recall some standard probabilistic tools and prove some preliminary results. In Section~\ref{section:matching} we use standard techniques to find bounds on $\nu(\Gour{m})$ for various time steps $m$ and prove Theorem~\ref{thm:finalgraph}\eqref{itm:pm}. Theorem~\ref{thm:finalgraph}\eqref{itm:me} is proved in Section~\ref{section:finalcover}. In Section~\ref{section:cover} we prove Theorem~\ref{thm:finalgraph}\eqref{itm:ovc}. We also show that, at almost all time steps from $c_1 n \log n $ to $2n \log n$, the number of vertices contained in an optimal cover is close to $\frac{n}{2}$.  This will be used in Section~\ref{section:delay} to prove Theorem~\ref{thm:delayedpm}. We conclude in Section~\ref{section:open} by mentioning some related open problems. 


\section{Overview of Proof}

In this section we give a brief overview of the arguments to come, introduce our main lemmas and define some notation that will be used throughout. 
Throughout the proof $n$ will be taken to be sufficiently large when needed. Any statements made during the discussion in this section about $\Gour{m}$ hold with high probability (but we may not write this every time). Of course, any formal statements are made explicit. 

Throughout the paper, the probability measure conditional on $(e_1, \dots, e_m)$ is $\mathbb{P}_m$, and the corresponding expectation is $\mathbb{E}_m$. In a slight abuse of notation, we use $(\Gour{m})$ for the probability space as well as the sampled process. The logarithm to base $e$ is denoted by $\log$.

Our first goal in Section~\ref{section:matching} is to show that, with high probability, when $m$ is not too small, $\Gour{m}$ contains a matching of size $\frac{n}{2}(1 - o(1))$. More precisely, we will prove the following.

\begin{lem}  \label{lem:betterpm}
	Let $\gamma >0$ be a sufficiently large constant and let $ \gamma n \le m \le 2n\log n$. With probability $1 - O(n^{-2}),$ 
	$$\nu(G_m) \ge \frac{n}{2} \left(1-e^{-3m/(10n)}\right).$$ 
\end{lem}
This lemma will follow from the fact that each vertex is typically in linearly many acceptable pairs and  fairly standard graph-theoretic arguments (see, for example, \cite[Section 6.1]{fk16}) relying on expansion properties of the graph. The same reasoning yields that, with high probability, after $m:= (1 + \epsilon)n \log n$ steps, each vertex has degree at least $\Omega(\log n)$ and  $G_{m}$ has a perfect matching. This proves statement~\eqref{itm:pm} of Theorem~\ref{thm:finalgraph}.

In Section~\ref{section:finalcover} we prove Theorem~\ref{thm:finalgraph}\eqref{itm:me}. Let $\cfinal$ be an optimal cover in $\Gour{N}$ and  let $D_m \subseteq V(\Gour{m})$ be the set of vertices contained in some optimal cover of $\Gour{m}$. How can an edge $e$ incident to $\cfinal$ be rejected during the process?  This is only possible if $e$ is not incident to a current optimal cover at the time step $m$ when it is offered (i.e.~$e \cap D_m = \emptyset$). The key idea in controlling those rejected edges is  to show that for $m' = 1.1 n \log n$, most vertices in an optimal vertex cover of $\Gour{m'}$ have been in an optimal cover for most $m<m'$. Therefore we rarely rejected an edge touching such a vertex. 
We also know that with high probability, $\Gour{m'}$ has a perfect matching. Then deterministically,  no edges incident to $\cfinal$ will be rejected after the time step $m'$, since $C_N$ is also an optimal cover in $\Gour{m'}$.

The question considered in Section~\ref{section:cover} is how `rigid' an optimal cover of $\Gour{m}$ is during the earlier evolution of the process. We will see that if $|D_m|$ is significantly larger than $\frac n2$ in a positive proportion of steps $m$, we are accepting \emph{too many} edges into our graph to maintain an independent set of order $\frac n2$.
Let us state the lemma formally. For $m_0,t \ge 0$, let $T(m_0,t)$ be the time period of length $t$ beginning at $m_0$, i.e.~$T(m_0,t):= \{m_0,\ldots, m_0 + t-1\}$.
 
 \begin{lem}
\label{lem:coverfrozen}
	Let $m_0 := n \sqrt{\log n}$. For $\gamma > 0$, let $T:= T(m_0,\gamma n \log n)$. With high probability, there are at most $\gamma n (\log n)^{9/10}$ values of $m \in T$ such that $|D_m| > \left(\frac{1}{2} +   (\log n)^{-1/10}\right)n$.  
\end{lem} 

This lemma will be useful in proving Theorem~\ref{thm:delayedpm}, as it gives us an upper bound on the number of acceptable pairs for $\Gour{m}$. We conclude Section~\ref{section:cover} by showing that after $4 n \log n$ steps, the optimal cover in $\Gour{m}$ is unique, which proves Theorem~\ref{thm:finalgraph}\eqref{itm:ovc}. 

The proof of Theorem~\ref{thm:delayedpm} is given in Section~\ref{section:delay}. Let us discuss roughly how the number of isolated vertices decays in the K\H{o}nig process $(G_m)$ compared to  a typical Erd\H{o}s-R\'enyi process $(G_{n,m})$.   In a single time step of $(G_{n,m})$, each isolated vertex is connected with probability asymptotically $2/n$. This dynamic allows one to relate the number of isolated vertices in $G_{n, m}$ to the classical coupon-collector problem and implies that with high probability, the isolated vertices disappear when $m = \frac 12 (n \log n + \omega(1))$.

Now consider the K\H{o}nig process $(\Gour{m})$ after $\nu(G_m)\geq (\frac 12 + o(1))n$. Let $I_m$ be the number of isolated vertices in $\Gour{m}$. We will show that there are time steps (called \emph{unhelpful}) in which each isolated vertex is connected with probability only $\frac 1n(1+o(1))$. The reason is that in an unhelpful step $m$, an isolated vertex $v$ cannot be used to extend a maximum matching in $G_m$, so now consider the K\H{o}nig process (roughly half of the vertex pairs containing $v$ (those not incident to a current optimal cover) are unacceptable. This is shown in Subsection~\ref{subsec:zerohelpers} using a careful analysis of a maximum matching in $G_m$.
In Subsection~\ref{subsec:par} we show that the unhelpful states typically constitute a constant proportion of the time steps. Since in an unhelpful state $m$, $I_m$ has a smaller decay rate than in $G_{n,m}$, it follows (via martingale arguments in Subsection~\ref{subsec:mart}) that $\Gour{\left( 1/2 + C \right) n \log n}$ has isolated vertices for some $C >0$. Our argument does not give a sharp constant $C$, so we do not attempt to make marginal improvements.

\section{ Preliminaries and Probabilistic Tools } \label{section:prelim}
In this section we gather together some basic probabilistic tools and standard results that will be used throughout the paper.
\subsection{The relationship between $\Gnm$ and $\Gnp$}
Let $\Gnp$ denote the $n$-vertex random graph in which every possible edge is present independently with probability $p$.
The following lemma allows us to prove that $\Gall{m}$ has certain properties, by considering the properties of $G(n,m/N)$
. A graph property $\mathcal{P}$ is said to be \emph{monotone increasing} if $G \in \mathcal{P}$ implies that $G + e \in \mathcal{P}$. 

\begin{lem}[\cite{fk16}, Lemma 1.3] \label{lem:gnpisfine}
	Let $\mathcal{P}$ be any graph property and let $p = p(n)$ satisfy $n^2p, n(1-p)p^{-1/2} \rightarrow \infty$. If $m = p \binom n2$ and $n$ is sufficiently large, then
		$$\pr{\Gnm \in \mathcal{P}} \le 10 m^{1/2} \pr{\Gnp \in \mathcal{P}}.$$	
        Moreover, if $\mathcal{P}$ is  monotone increasing, then $\mathbb{P}(\Gnm \in \mathcal{P}) \le 3\mathbb{P}(\Gnp \in \mathcal{P})$. 
\end{lem}

As a consequence, we immediately get a bound on the maximum degree in $\Gall{10n \log n}$.

\begin{claim}\label{maxdeg}
For $m = 10n \log n$, the graph $\Gall{m}$ has maximum degree at most $200 \log n$ with probability $1-O\left(n^{-2}\right)$. 
\end{claim}
\begin{proof}
Let $p:= \frac{10n \log n}{\binom{n}{2}}$. Using the Chernoff bound (the third formulation in Theorem~\ref{Chernoff} stated below), the probability that a single vertex has degree $200 \log n$ in $\Gnp$ is at most $2^{-200 \log n}$. Taking the union bound over all $n$ vertices and applying Lemma~\ref{lem:gnpisfine} gives the required result. 
\end{proof}

\subsection{Standard Estimates and Probabilistic Tools}

Here we collect together the standard probabilistic tools we will use during the proof. The first is a version of the Chernoff Bound taken from~\cite{dp09}. 

\begin{thm}[The Chernoff Bound]

\label{Chernoff}
Let $X_1,\dots,X_n$ be a sequence of independent $[0,1]$-valued random variables and let $X=\sum_{i=1}^nX_i$. Then, for $0<\varepsilon<1$, 
\[\mathbb{P}\left(X<(1-\varepsilon)\mathbb{E}(X)\right)\leq e^{-\frac{\varepsilon^2\mathbb{E}(X)}{2}},\]
\[\mathbb{P}\left(X>(1+\varepsilon)\mathbb{E}(X)\right)\leq e^{-\frac{\varepsilon^2\mathbb{E}(X)}{3}}.\]
Moreover, if $t > 2e \er{X}$, then 
$$\pr{X > t} \leq 2^{-t}.$$
\end{thm}

 We will use a well-known result about the edge distribution in the random graph. As usual,  $E_G(U, W)$ denotes the set of edges of a graph $G$ with one endpoint in $U$ and one in $W$ and $E_G(U) $ the set of edges with both endpoints in $U$. We omit the index $G$ when the graph is clear from the context. The number of edges in $G$ is denoted by $e(G)$.  This particular form is stated in \cite{ks06} for $\Gnp$, but Lemma \ref{lem:gnpisfine} implies that it also holds for $\Gnm$.

\begin{thm} \label{thm:edgedistribution}
	Let $m = m(n) \leq 0.49n^2$. There exists a constant $\lambda >0$ such that, with high probability, in $\Gnm$ every two disjoint sets $U, W \subset V$ of cardinality $|U| = u$, $|W |= w$ satisfy
    $$ |E(U, W)| \in \frac{2muw}{n^2} \pm \lambda\sqrt{\frac{uwm}{n}} \text{\quad and \quad }
     |E(U)| \in \frac{mu^2}{n^2} \pm O\left(u\sqrt{\frac{m}{n}} \right).$$ 
\end{thm}

We need another claim on the edge distribution in $\Gnm$. Although similar results are available in the literature, we will need an explicit bound on the probability so we include the proof here.

        \begin{lem} \label{lem:upedge}
        	Let $\beta \leq 0.1$ and $m \geq 10 \beta^{-2}n$. For $G:=\Gnm$ and any set $U$ of cardinality  at most $\frac{\beta n}{2e}$, we have
            $$\pr{|E_G{(U)}|\leq \frac{\beta m |U|}{n}} = 1- O(n^{-2}).$$
        \end{lem}
     
        \begin{proof}
        	We say that $G \in \mathscr{D}_u$ if some set $U \subset V(G)$ of cardinality $u$ spans more than $\frac{\beta m u}{n}$ edges in $G$, and set $\mathscr{D} = \bigcup_{u = 1}^{\beta n/ (2e)} \mathscr{D}_u$. We will show that $\pr{\Gnp \in \mathscr{D}} = O\left(n^{-5} \right)$ in $\Gnp$ with $p = \frac{2m}{n^2}$. 
           
            We estimate the probability of $\mathscr{D}_u$ by taking the union bound over all the vertex sets of cardinality $u$.
            \begin{align*}
            	\pr{\Gnp \in \mathscr{D}_u}\leq \binom nu \binom {\frac{u(u-1)}{2}}{\frac{\beta np u}{2}} p^{\frac 12 \beta np u}
                \leq \left (\frac{en}{u} \left(\frac{eu}{\beta n} \right)^{\frac 12 \beta np} \right)^u.
            \end{align*}
    We proceed by splitting the range for $u$. First note for $u < \beta np$, $\mathscr{D}_u$ is empty as a set of cardinality $u  $ cannot span more than $\frac{\beta np u}{2}=\frac{\beta m u}{n}$ edges. If $u \leq \sqrt{n}$ and $n$ is sufficiently large, then $eu/(\beta n)\leq n^{-1/3}$, so
    $$\pr{\Gnp \in \mathscr{D}_u} 
    \le  \left(e n^{\frac 12}\cdot n^{-\frac 13 \cdot \frac 12 \beta np}\right)^u
   \le \left(e n^{\frac 12 - \frac 16 \beta np} \right)^u < n^{-2u}.$$
    For $\sqrt{n}< u \leq \frac{\beta n}{2e}$, we recall the hypothesis that $\frac 12 \beta np \geq 10 \beta^{-1}$ and deduce 
    $$\pr{\Gnp \in \mathscr{D}_u} \le \left(\frac{en}{u} \left( \frac{eu}{\beta n} \right)^{\frac 12 \beta np} \right)^u
    \leq  \left (2e^2\beta^{-1} \cdot 2^{-10\beta^{-1}} \right)^u < 2^{-u}.$$
    Combining these estimates, we get
            	$$\pr{\Gnp \in \mathscr{D}} \leq \sum_{u = \beta np}^{\sqrt{n}} n^{-2u} + 
                \sum_{u = \sqrt{n}}^{\frac{\beta n}{2e}}2^{-u} = O\left(n^{-\beta np} \right) + O \left(n \cdot 2^{-\sqrt{n}}\right) 
                = O \left( n^{-5} \right).$$
                Lemma \ref{lem:gnpisfine} gives $\pr{\Gnm \in \mathscr{D}}  =O \left( n^{-2} \right).$ 
        \end{proof}

The following lemma is a bound for the lower tail of the binomial distribution. If we are looking to control very large deviations, elementary computations give a stronger estimate than the usual Chernoff bounds.

\begin{lem} \label{lem:EltCalc}
	Let $X \sim \textnormal{Bin}(k, p)$ with $\mu = kp \rightarrow \infty$ as $k \rightarrow \infty$. Given $\eta > 0$ and a constant $0 < \delta \le 1$ satisfying $\delta(2+ \log(1/\delta))+ (\log(\delta \mu))/\mu < \eta$ for sufficiently large $k$, we have
    $$\pr{X \leq \delta \mu} \leq e^{-(1-\eta)\mu}.$$

\end{lem}
 \begin{proof} 
 	By definition of $\Bin(k, p)$, $\pr{X \leq \delta \mu} \leq \sum_{s=0}^{\delta \mu} \binom ks p^s(1-p)^{k-s}$. Standard inequalities give 
    \begin{align*}
    	\pr{X \leq \delta \mu} 
        \leq \sum_{s=0}^{\delta \mu} \left(\frac{ekp}{s} \right)^s e^{-(k-s)p} 
        = \sum_{s=0}^{\delta \mu} \left(\frac{e^{1+p}\mu}{s} \right)^s e^{-kp},
    \end{align*}
    As $\delta \le 1$ and the summand is increasing in $s$ for $s \le \mu$, we can bound each summand by the final term, $s = \delta \mu$. It follows that
    \begin{align*}
    	\pr{X \leq \delta \mu} 
        \leq \delta \mu \left(e^2 \delta^{-1} \right)^{\delta \mu}e^{-\mu}.
    \end{align*}
    	The hypothesis on $\delta$ is equivalent to
        $\log(\delta \mu)+ \delta \mu (2 + \log(1/ \delta)) \leq \eta \mu$, and hence
        $$\pr{X \leq \delta \mu} \leq e^{\eta \mu}e^{-\mu}.$$
 \end{proof}

\subsection{Martingale concentration inequalities}
Recall that a sequence of random variables $X_0,X_1,\ldots$ is called a \emph{martingale} if for each $i \ge 1$, we have $\mathbb{E}[X_i\given X_0,\ldots,X_i] = X_{i-1}$. It is called a \emph{submartingale} if $\mathbb{E}[X_i\given X_0,\ldots,X_i] \ge X_{i-1}$. We will use two standard martingale concentration results in our proof. The first is Azuma's Inequality. The version we present here was taken from~\cite{dp09}. 

\begin{thm}[Azuma's Inequality] \label{thm:azuma}
Let $X_0,X_1,\ldots$ be a martingale such that for each $i \ge 0$ there exists a constant $c \ge 0$ such that $|X_m - X_{m-1}| \le c$. Then,
$$\pr{|X_m - X_0| > \lambda} \le 2 \exp\left(-\frac{\lambda^2}{2c^2m}\right).$$
\end{thm}

The second is Freedman's Inequality. It gives a stronger concentration result than Azuma's inequality when the  \emph{average} differences $X_{m+1}-X_m$ are much smaller than the worst-case. To avoid working with filtrations and the corresponding notation, we state it in our specific context. The general statement can be found for instance in \cite{tropp}.
Moreover, in Section~\ref{section:delay}, we will apply a version of the Inequality for submartingales, where the bound only holds for the lower tail of $Y_i$. This formulation (with slightly stronger constants) can be found in~ \cite{fan}.
        
        \begin{thm}[Freedman's inequality] 
        	\label{thm:freedman} 
             Let $\sum_{i=1}^{m}Y_i$ be a real-valued submartingale, where $\mathbb{E}_{i-1}[Y_i] \geq 0$ and $|Y_i| \leq 1$ for $i \geq 0$. Let $W_{m} = \sum_{i=1 }^{m} \erm{i-1}{Y_i^2}$. Then, for all $\ell >0$ and $\sigma^2 >0$,
            \begin{equation*}
                \pr{\exists m: \sum_{i=1}^m Y_i \leq -\ell \text{ and } W_m \leq \sigma^2 } \leq \exp \left[-\frac{\ell^2}{2 \sigma^2 + 4\ell /3} \right].
            \end{equation*}
            If, in addition, $\erm{i-1}{Y_i} = 0$ for $i \in \N$, 
            \begin{equation*}
                \pr{\exists m: \left |\sum_{i=1}^m Y_i \right | \geq \ell \text{ and } W_m \leq \sigma^2 } \leq 2\exp \left[-\frac{\ell^2}{2 \sigma^2 + 4\ell /3} \right].
            \end{equation*}
        \end{thm}

Moreover, the following special case of Theorem~\ref{thm:freedman} for indicator random variables will be useful in our applications. 

            \begin{cor} \label{cor:freedman}
                Let $X_i$ be an indicator random variable depending only on $\Gour{i}$ and let $q_i : = \prm{i-1}{X_i=1}$. For any $m \in \mathbb{N}$ and $L>0$,
                $$\pr{\left( \left| \sum_{i=1 }^m (X_i - q_i) \right | \geq Ln \right) \land \left(\sum_{i=1 }^m q_i \leq 4Ln \right)} \leq e^{-Ln/10}.$$
            \end{cor}
 
            \begin{proof}
                 Fix $m$, and define $T:=\{1, \dots m \}$. Let $Y_i: = X_i - q_i$ for $i \geq 1$. By definition, $\erm{i-1}{Y_{i}} = 0 $, so $ \left( \sum_{i=1}^{m'} Y_i \right)_{m'}$ is a martingale with $|Y_i| \leq 1$. Moreover, $\sum_{i \in T} {\erm{i}{Y_{i+1}^2}} = \sum_{i \in T}{q_i}(1- q_i)$, so our event implies the event stated in Freedman's inequality with $\sigma^2 = 4Ln$. 
                 Applying Freedman's inequality (with $\ell=L n$) gives precisely
                $$\pr{\left( \left| \sum_{i \in T} (X_{i} - q_{i}) \right | \geq Ln \right) \land \left(\sum_{i \in T} q_i \leq 4Ln \right)} \leq 2\exp \left[ -\frac{Ln}{8 + 4/3 } \right] \leq e^{-Ln/10}.$$
            \end{proof}

\section{Forming a large matching} \label{section:matching}

Our goal in this section is to prove Lemma~\ref{lem:betterpm}. It can be viewed as an analogue of Frieze's result~\cite{frieze86} on $\Gnm$ where he showed that for $m = \Theta(n)$, $\Gnm$ contains a matching covering almost all non-isolated vertices with high probability. 

We start by proving Proposition~\ref{prop:deterministic}, a deterministic property of the K\H{o}nig process. It will easily follow that after $O(n)$ steps we have a matching of linear size (see Corollary~\ref{cor:linmatch}). A little more work (in Lemma~\ref{lem:precisedegrees}) shows that we have a linear sized subgraph with `large' minimum degree satisfying certain expansion properties. We are then able to conclude the proof of Lemma~\ref{lem:betterpm} using a standard expansion argument (see, for instance, Section 6.1 of \cite{fk16}).  

	\begin{proof}[Proof of Proposition~\ref{prop:deterministic}]
    	By induction on $m$. The statement is trivial for $m=0$. Suppose that it holds for the first $m-1$ steps and consider the edge $e_m$. Let $M$ be a matching where the set of vertices covered by $M$, denoted $V(M)$, is incident to $\{e_1, e_2, \dots, e_{m-1} \}$. If $e_m$ is disjoint from $V(M)$, we can take $M+e_m$ 
to be the required optimal matching. If $e_m$  is incident to $V(M)$ and 
$\nu(\Gour{m})=\nu(\Gour{m-1})$, then $M$ still satisfies the condition.

Finally, if $e_m$ is incident to $V(M)$ and $\nu(\Gour{m})=\nu(\Gour{m-1})+1$, let $M'$ be a matching of size $|M|+1$ in 
$\Gour{m}$. The union of $M$ and $ M'$ consists of cycles and paths alternating 
between $M$ and $M'$, where one 
path $P$ containing $e_m$ has odd length, contains one more edge of $M'$ than of $M$, 
and two endpoints outside $M$;  all the remaining paths have even length. Then the matching $M_1$ created by 
replacing edges of $M$ in $P$ by those of $M'$, is a maximum matching in 
$\Gour{m}$ covering all vertices of $M$. Hence $M'$ intersects every edge in $\{e_1, \dots, e_m \}$, as required.
    \end{proof}
	
    \begin{cor}\label{cor:linmatch} Let $\epsilon >0$ and $m_0:= \epsilon^{-2} n$. With probability $1 - \Ont$
		\begin{enumerate}  
			\item \label{itm:linmatch1} $\nu \left( \Gour{m_0} \right) \geq \frac{1-\epsilon}{2} \cdot n $, and
			\item \label{itm:linmatch2} whenever $ m_0 \le m \le 2n \log n$, each vertex is contained in at least $\left(\frac 12 - \epsilon \right)n$ pairs which are acceptable for $\Gour{m}$.
		\end{enumerate}
	\end{cor}
    \begin{proof}
    	For every $m$, $\Gall{m}$ has an independent set of order $n- 2\nu(\Gour{m}) $, namely the complement of $V(M)$, where $M$ is the maximum matching granted by Proposition~\ref{prop:deterministic}. Denoting the order of a maximum independent set in $G$ by $\alpha(G)$, standard first-moment computations (see, e.g., \cite[Section 7]{jlr}) yield
        $$\pr{\alpha\left(\Gall{m_0} \right) \geq \epsilon n } \leq e^{-\epsilon n} = \Ont .$$
        It follows that with probability $1 - \Ont$, $n- 2\nu(\Gour{m_0}) < \epsilon n$, as required for \ref{itm:linmatch1}. 
        
        This tells us that with high probability $\tau(\Gour{m}) \ge \left(\frac 12 - \frac{\epsilon}{2} \right)n$ for $m \geq m_0$. Any pair containing $v$ and a vertex in an optimal vertex cover $C_m$ of $\Gour{m}$ is acceptable for $\Gour{m}$. By Corollary~\ref{maxdeg}, with probability $1 - O(n^{-2})$, $\Gall{m}$ has maximum degree $O(\log n)$ which implies that $o(n)$ pairs incident to $v$ have been offered so far. Thus the number of acceptable pairs incident to $v$ is at least  $\left(\frac 12 - \epsilon \right)n$, as required for \ref{itm:linmatch2}.
    \end{proof} 
    
    To avoid confusion, we remark that the bound in \ref{itm:linmatch1} is rather crude since the independence number of $\Gnm$ with $m = \epsilon^{-2} n$ is actually  $\Theta \left( \epsilon^2 n\log \left( \epsilon^{-1} \right) \right)$. The probability `benchmark' $\Ont$ across this section is also arbitrarily chosen -- all the probability bounds are significantly stronger.

	Even though we need a stronger bound on $\nu(\Gour{m})$,  Corollary~\ref{cor:linmatch} is a very useful tool, providing a lower bound on the probability that $e_{m+1}$ is acceptable for $\Gour{m}$. As usual, we let $\delta(G)$ denote the minimum degree of $G$. The neighbourhood of a vertex set $S$ in a graph $G$, excluding $S$, is denoted by $N_G(S)$. We may omit the subscript when it is clear which graph plays the role of $G$.
The following facts about the edge distribution of $\Gour{m}$ will be used for our expansion arguments.

	\begin{lem}		\label{lem:precisedegrees}
		Let $0< \beta \le 10^{-3}$. There exists $\gamma>0$ such that for all $\gamma n \leq m \leq 2n \log n$, with probability at least $1- \Ont$, $\Gour{m}$ has the following properties.
		\begin{enumerate}[(i)] 
		\item \label{itm:precisedegrees1} $\Gour{m}$ has a subgraph $H$ such that $|V(H)| \ge n\left(1 - e^{-m/(3n)} \right)$ and $\delta(H) \ge \frac{10 \beta m}{n}$.
		\item \label{itm:precisedegrees2} For any set $S \subseteq V(H)$ with $|S| \le \frac{\beta n}{16}$,   $|N_H(S) \cup S| > 2|S|$.
		\end{enumerate}
	\end{lem}

	\begin{proof}
	     First consider \ref{itm:linmatch1}. We will use the following claim. 
	    
	    \begin{claim} \label{claim:lowerbounddegree}
	        Let $\alpha:= e^{-m/(3n)}$ and let $\mathcal{A}$ be the event that there exists a set $S$ of order $\alpha n$ such that $|E_{\Gour{m}}\left(S, V \setminus S \right)| \leq 10 \alpha \beta m$. Then $\pr{\mathcal{A}} = \Ont.$
	    \end{claim}

	    \begin{proof}
	    	Let $\mathcal{B}$ be the event that $\Gall{m}$ has maximum degree at most $200 \log n$ and the statements \ref{itm:linmatch1} and \ref{itm:linmatch2} of Corollary~\ref{cor:linmatch} hold with $\epsilon = \frac{1}{20}$. By Claim~\ref{maxdeg} and Corollary~\ref{cor:linmatch}, we can pick $\gamma$ large enough to ensure that $\mathcal{B}$ occurs with probability at least $1 - O\left(n^{-2} \right)$. It is now sufficient to show that
	    	   	\begin{equation}\label{eq:edgeset}
		    	   	\pr{\mathcal{A} \given \mathcal{B}} \le n^{-2}.
	    	   	\end{equation}
	    	
			Let $S$ be a set of order $\alpha n$. As we condition on $\mathcal{B}$ occurring, we can choose $\gamma/10$ to be sufficiently large such that for all $m/10 \le m' \le 2n\log n$, each vertex in $G_{m'}$ is contained in at least $\frac{9}{20}n$ pairs acceptable for $G_{m'}$. 	Also, by choosing $\gamma$ to be sufficiently large, we can ensure that $|S| \leq \frac{n}{100}$. 
			
			Combining these two facts with the maximum-degree bound gives that for all $m/10 \le m' \le 2n \log n$, each $v \in S$ is in at least
			 $\frac{2n}{5}$ acceptable pairs $(v,w)$ for $G_{m'}$ such that $w \notin S$. So the probability that $e_{m'+1}$ is acceptable and has exactly one endpoint in $S$ is at least 
			 	$$\frac{\frac{2n}{5}|S|}{\binom{n}{2}} \ge \frac {4\alpha}{5}.$$
			Hence 
    	    \begin{align*}
    	        \pr{|E_{\Gour{m}}\left(S, V \setminus S \right)| \leq 10 \alpha \beta m \given \mathcal{B} }
    	        & \leq \pr{\Bin \left(\frac{9m}{10}, \frac{4\alpha}{5} \right) \leq 10 \alpha \beta m}
    	         \leq  e^{-\frac 12 \alpha m }.
    	    \end{align*}
    	    The last inequality follows from Lemma \ref{lem:EltCalc} (applied with $\mu = 36\alpha m/50$, $\eta = 11/36$ and $\delta = 1/72$).
            Taking the union bound over all sets of order $\alpha n$ 
            \begin{align*}
                \pr{\mathcal{A} \given \mathcal{B}} \leq \binom{n}{\alpha n}\cdot e^{-\frac 12 \alpha m}
                \leq \left(e \alpha^{-1} e^{-\frac{m}{2n}} \right)^{\alpha n} = \left( e^{-\frac{m}{6n}+1} \right)^{\alpha n}  \leq e^{-\frac{\alpha m}{10}}.
            \end{align*}
            Since $n \leq m \leq 2n \log n$, a very crude computation gives $\alpha m = e^{-m/(3n)} m\geq e^{-2\log n / 3} m \geq n^{1/3}$, so  $\pr{\mathcal{A}} \leq e^{-\alpha m / 10} \leq n^{-2}$,
            completing the proof of (\ref{eq:edgeset}) and hence the proof of the claim. 
        \end{proof}

        Now consider applying the following algorithm to $G_m$. Let $H_0:= G_m$ and $R_0:= \emptyset$. Now for each $i \ge 0$, if there exists some $v \in H_i$ such that $\deg_{H_i}(v) <  \frac{10 \beta m}{n}$, then define $H_{i+1}:= H_i \setminus \{v\}$ and $R_{i+1}:= R_i \cup \{v\}$. We terminate the algorithm when no such $v$ exists, and denote the final step by $j$. Claim~\ref{claim:lowerbounddegree} implies that $\pr{j \geq \alpha n}  = \Ont$.  For, if $R_{\alpha n}$ is defined, then  $|E_{G_m}(R_{\alpha n}, V \setminus R_{\alpha n} )| \le  10 \alpha \beta m$ by construction. This completes the proof of (\ref{itm:precisedegrees1}).
        
       To show (ii), assume that part (i), as well as the conclusion of Lemma~\ref{lem:upedge} hold. Let $H$ be the subgraph given by (i), and  let $S\subseteq V(H)$ be a set of cardinality $|S|\leq \frac{\beta}{16}n$. Let $T = N_{H}(S) \cup S$ and suppose that $|T| \le 2|S| \le \frac{\beta}{8}n$. By Lemma~\ref{lem:upedge} and the minimum-degree assumption about $H$, we have
       		$$\frac{\beta m}{n}|T| \ge |E_{\Gour{m}}(T)| \ge \frac{10 \beta m}{2n}|S|,$$
       and so $|T| \ge 5|S|$, a contradiction. So $|T| > 2|S|$, as required for (ii).
	\end{proof}

	We are now ready to present the proof of Lemma~\ref{lem:betterpm}.

	\begin{proof}[Proof of Lemma~\ref{lem:betterpm}]
	
			Let $0< \beta \leq 10^{-3}$  and $\zeta:= 2^{10}\beta^{-2}$. Given $m \geq \gamma n$ and $\gamma$ sufficiently large, define $m' = m - \zeta n \geq \frac {9m}{10}$. Let $\mathcal{B}$ be the event that statements \eqref{itm:precisedegrees1} and \eqref{itm:precisedegrees2} of Lemma~\ref{lem:precisedegrees} hold for $m'$ and let $\gamma$ be large enough to ensure $\mathbb{P}(\mathcal{B}) = 1 - \Ont$ (possible by Lemma~\ref{lem:precisedegrees}). We show that
			\begin{equation}\label{eq:match}
					\pr{ \left. \nu(G_m) < \frac{n}{2} \left(1-e^{-3m/(10n)}\right) \, \right \rvert \mathcal{B}}  = \Ont,
			\end{equation}
		which will imply the lemma. In what follows, we condition on $\mathcal{B}$ occurring. 
		
		Hence there exists a set $L$ of vertices which span a subgraph of minimum degree at least $\frac{10 \beta m'}{n}$ in $G_{m'}$, such that $n - |L| \leq ne^{-m'/3n} \leq ne^{-3m / 10n}$. We may ensure $L$ has an even number of vertices by removing an arbitrary vertex if $|L|$ is odd. Define $H_i := G_{i}[L]$ for each $i\geq m'$. We use an expansion argument due to Bollob\'as and Frieze~\cite{bf85} to show that $H_{m}$ contains a perfect matching. We will include the proof since the claim is not stated explicitly in~\cite{fk16}.
 
        \begin{claim}[\cite{fk16}] \label{claim:fkmatching}
        	Let $H$ be an $n$-vertex graph in which every $S \subseteq V(H)$ with $|S| \leq k$ satisfies  $|N_{H}(S)| > |S|$. If $H$ does not have a perfect matching, then there are at least $\binom{k }{2}$  vertex pairs $f$ such that $\nu(H + f) > \nu(H)$.  
        \end{claim}
		\begin{proof}
		    Let $M$ be a maximum matching in $H$ and $v$ a vertex not contained in $V(M)$. In this case, we say that $M$ \emph{isolates} $v$.
		    Using a sequence of \emph{flips}, we will find many vertex pairs $uv$ which extend $M$. For $e  = xx' \in M$ and $f = xu$, where $u \neq v$ is isolated by $M$, a \emph{flip} from $e$ to $f$ is the operation of replacing $M$ by $M' = M+f-e$. Let $A(v)$ be the set of vertices $y$ such that $v$ and $y$ are isolated by some matching $M'$ obtained from $M$ by a sequence of flips. 
		    
		    We will show that any vertex in $N_H(A(v))$ is matched in $M$ to some vertex of $A(v)$, implying that $|N_H(A(v))| \leq |A(v)|$. Let $x \in N_H(A(v))$. In particular $x \notin A(v)$ and so $x \in M$. 
		    Let $x'$ denote the unique neighbour of $x$ in $M$. 
		    We will show the vertex $x'$ is in $A(v)$. This will imply that $|N_H(A(v))| \leq |A(v)|$. Since $x \in  N_H(A(v))$ there is a vertex $u \in A(v)$ adjacent to $x$. 	    
		    Let $M'$ be a maximum matching obtained from $M$ by a sequence of flips that isolates $u$. First, suppose that $xx' \in M'$. Then we can flip the edge $xx'$ with the edge $xu$, isolating $x'$. Thus $x' \in A(v)$. If $xx'$ is not in $M'$,
		    then at some point in the sequence of flips from $M$ to $M'$ a flip from $xx'$ to another edge has occurred. If this happens, then either $x$ or $x'$ is in $A(v)$. By assumption, $x \notin A(v)$, so in fact $x' \in A(v)$. 
		    Hence $|N_H(A(v))| \leq |A(v)|$, as required.
		    
		    Our assumption on $H$ in the claim hypothesis implies that $|A(v)| \geq k$. Applying the same argument to all vertices $x \in A(v)$ (which are isolated by some maximum matching $M'$) gives us at least $\frac{k^2}{2}$ distinct pairs $xy$ such that $\nu(H + xy) > \nu(H)$.
		\end{proof}
			Consider any step $i$ with $m' < i \leq m$ and $\nu(G_i) < \frac n2$. As we condition on $\mathcal{B}$, we may apply Claim~\ref{claim:fkmatching} with $k = \frac{\beta}{16}n$ to get that there are at least $\binom{\beta n/16}{2}$ vertex pairs $f$ such that $\nu(\Gour{i} + f) > \nu(\Gour{i})$. As $i \le 2n \log n$, only $o(n^2)$ vertex pairs have been offered so far. Therefore the probability that $\nu(\Gour{i+1}) > \nu(\Gour{i})$ is at least the probability that such an $f$ is offered, which is
				\begin{equation}\label{eq:incrmatch}
				\frac{\binom{\beta n/16}{2} - i}{\binom{n}{2} - i} \ge \frac{\beta^2}{2^9}. 
				\end{equation}
			
			Let $Y \sim \Bin \left(\zeta n, \frac{\beta^2}{2^9} \right)$. We have 
				$$\pr{H_{m} \text{ has no perfect matching} \given \mathcal{B}} \le \pr{Y < \frac n2}= e^{-\Omega(n)} \leq n^{-2}.$$
        	This completes the proof of the lemma.  
         
    \end{proof}

Analogous arguments actually give a stronger result, which will be used to prove Theorem~\ref{thm:finalgraph}\eqref{itm:pm}. 
	\begin{cor} \label{cor:mindegreepm}
		Let $\delta> 0$, $m \gg n$. Let $L$ be a fixed vertex subset of even order. If $\Gour{m-\omega(n)}[L]$ has minimum degree at least $\frac{\delta m}{n}$, then $\Gour{m}[L]$ has a perfect matching with high probability.
	\end{cor}
\begin{proof}
    Let $L$ be a vertex subset of even order such that $H:= \Gour{m-\omega(n)}[L]$ has minimum degree at least $\frac{\delta m}{n}$. Then, analogously to the argument in \ref{lem:precisedegrees}\eqref{itm:precisedegrees2}, we obtain that, with high probability, any subset $S \subseteq V(H)$ with $|S| \le \frac{\delta n}{160}$ satisfies $|N_H(S) \cup S| > 2|S|$. 
    Now, as in the proof of Lemma~\ref{lem:betterpm}, using Claim~\ref{claim:fkmatching}  we see that $\Gour{m}[L]$ has a perfect matching with high probability.
    \end{proof}
    
To conclude the section, we show that $\Gour{m}$ has a perfect matching for $m =(1+\epsilon) n \log n $ to prove Theorem~\ref{thm:finalgraph}\eqref{itm:pm}. It is not difficult to show the claim for $m = (1+o(1))n \log n$ by controlling distances between `low-degree' vertices  in $\Gour{m}$ (see, e.g., \cite{fk16}), but we chose to include the slightly weaker statement, which is restated here for the benefit of the reader.

\begin{thmi}
    	Let $\epsilon >0$ and $m = (1 + \epsilon)n \log n$. With high probability, $\Gour{m}$ has a perfect matching.
    \end{thmi} 
    \begin{proof}
    	Let $t := \frac{\epsilon}{4} n \log n$. We will first show that with high probability, $\Gour{m-t}$ has minimum degree at least $\delta \log n$. 
        
        By Corollary~\ref{cor:linmatch}, if $t \leq m' \leq 2n \log n$, each vertex is in at least $\left(\frac 12 - o(1) \right)n >\left(\frac 12 - \frac{\epsilon}{40} \right)n$ vertex pairs which are acceptable for $G_{m'}$. Therefore, for all $k\ge 0$, we have
         	$$\pr{\deg_{\Gour{m-t}}(v) \le k} \le \pr{X \le k},$$
         where $X \sim\Bin\left(m-2t, \left(\frac{1}{n}-\frac{\epsilon}{20n} \right)  \right)$. Let $\delta \le 1$ be small enough for Lemma \ref{lem:EltCalc} to hold with $\eta = \frac{\epsilon}{20}$. By the choice of $\delta$ and applying Lemma~\ref{lem:EltCalc}, we have
        $$\pr{\deg_{\Gour{m-t}} (v) < \delta \log n} \leq 
        \exp \left(-\left(1- \frac{\epsilon}{20} \right) \left(m-2t \right) \left(\frac{1}{n}-\frac{\epsilon}{20n} \right) \right)
       	\leq n^{-1 - \frac{\epsilon}{10}}.$$
        Taking the union bound over all $n$ vertices gives that with high probability, $\Gour{m-t}$ has minimum degree at least $\delta \log n$.
        
        As $t = \omega(n)$, we can apply Corollary \ref{cor:mindegreepm} to $\Gour{m-t}$ with $L = [n]$ and deduce that $\Gour{m}$ has a perfect matching.
    \end{proof}

\section{The structure of $G_N$} \label{section:finalcover}
    The main purpose of this section is to prove Theorem~\ref{thm:finalgraph}\eqref{itm:me}. That is, we bound the number of rejected vertex pairs incident to an optimal cover of $\Gour{N}$. Let $\cfinal$ denote an optimal cover of $\Gour{N}$. Let $m = 1.1n\log n$. When $G_m$ has a perfect matching (which it does with high probability by Theorem~\ref{thm:finalgraph}\eqref{itm:pm}), $\cfinal$ will also be an optimal cover in $\Gour{m}$ and any edge incident to it will be accepted. So it suffices to control the edges of $\Gall{m}$ rejected from an optimal cover $C_m$ of $\Gour{m}$ when $m = 1.1 n \log n$. This is done in Lemma~\ref{lem:rejedge}. This lemma and the previous observation immediately imply Theorem~\ref{thm:finalgraph}\eqref{itm:me}. We note that, in particular, our results in this section do not rely on uniqueness of $\cfinal$ (Theorem~\ref{thm:finalgraph}\eqref{itm:ovc}).

    We start by introducing some concepts that will be used throughout the section. Recall that $D_m \subseteq V(\Gour{m})$ is the set of vertices contained in some optimal cover of $\Gour{m}$.
        For a time period $T:= T(m_1,t)$ and vertex $v \in V$, define the \emph{weight} $W_T(v)$  of a vertex as
        	$$W_T(v):= \lvert \{m \in T: v \in D_m\}\rvert.$$
       Note that $W_T(v)$ is a function of our random process. For a set $S \subseteq V$, define the \emph{average weight} of $S$ in $T$ as 
      		 $$W_T(S):= \frac{1}{|T|}\sum_{v \in S}W_T(v).$$
        
		The main ingredient in the proof of Lemma~\ref{lem:rejedge} is the following lemma (Lemma~\ref{lem:covercontrol}). It is proved using a martingale trick similar to one that will be used in Lemma~\ref{lem:lowerboundedges}. The main difference is that here we apply Freedman's inequality (Corollary~\ref{cor:freedman}), whereas there we apply Azuma's inequality. This is because Azuma's inequality considers the worst-case change of a martingale $(X_m)$. In our case the \emph{typical} changes are much smaller. Therefore Freedman's inequality gives a stronger bound, which is also necessary for the computations.
        
         \begin{lem}\label{lem:covercontrol}
        For $\gamma >0$, let $m_1 \ll t := \gamma n \log n$ and let $T:= T(m_1,t)$. The following holds with high probability.
        \begin{enumerate} 
            \item \label{itm:covercontrol1} No set $S$ of order at least $\frac n3$ with $W_T(S) \geq \frac{150n^2}{t}$ is independent in $\Gour{m_1+t}$.
            \item \label{itm:covercontrol2} Let $a \ge 10$. For any set $U$ of order $\frac n2$ with $W_T(U) \geq \frac n2 - \frac{an^2}{t}$, the number of edges incident to $U$ which were rejected during $T+1 := T(m_1+1, t)$ is at most $5an$.
        \end{enumerate}
        \end{lem}
        
    \begin{proof}[Proof of Lemma \ref{lem:covercontrol}]
        Let $m \ge m_1$ and let $S \subseteq V$ be a set of order at least $\frac n3$. For each $ m$, define $Q'_{m+1}$ to be the set of vertex pairs in $S^{(2)}$ which are acceptable for $\Gour{m}$. Let $s$ be the maximum integer such that      
        \begin{equation}\label{Qtrunc}
        \sum_{i =m_1+1}^{s}\frac{|Q'_{i}|}{\binom {n}{2}-i+1}\leq 79n.
        \end{equation}
        For $m_1 + 1 \le j \le s$, define $Q_{j}:= Q'_j$, and for $j > s$ define $Q_j := \emptyset$. The reason we truncate the sequence $(Q'_j)_{j \ge m_1 + 1}$ in this manner is to deal with a technicality in our application of Freedman's inequality.
        
        Let $X_m$ be the indicator random variable of the event that $e_m \in Q_m$. Moreover, define 
        $$q_{m+1} := \prm{m}{X_{m+1} = 1 } = \frac{|Q_{m+1}|}{\binom {n}{2}-m},$$
        so that by definition we have
        \begin{equation*}\label{qmup}
        	\sum_{m \in T}q_{m+1} \le 80n.
        \end{equation*}

Given this, by applying Corollary~\ref{cor:freedman} with $L = 20$, we see that
		\begin{equation}
        \label{freed1}
        \pr{\left|\sum_{m \in T}(X_{m+1} - q_{m+1})\right| \ge 20n} \le e^{-2n}.
        \end{equation}
        Let $\mathcal{A}$ be the event that $W_T(S) \geq \frac{150n^2}{t}$. Let $\mathcal{B}$ be the event that $X_{m+1}=0$ for all $m \in T$. We will show that $\pr{\mathcal{A} \land \mathcal{B}} \le e^{-2n}$.   
        
        To use the condition on $W_T(S)$, we will need a simple relation between $q_m$ and $W_T(S)$. 
            \begin{claim} \label{claim:ProbExp}
                 If $\mathcal{A}$ occurs, then $ \sum_{m \in T} q_{m+1} \geq 40n$.
            \end{claim}
                \begin{proof} 
                 Any vertex pair not contained in $\Gall{m}$ but intersecting $S \cap D_{m}$ is acceptable for $\Gour{m}$. Therefore, we have
                $
                 |Q_{m+1}| \geq \frac 12 |S \cap D_{m}|(|S|-1)-m.
                $
                Summing over $m$ and using $\sum_{m \in T} |S \cap D_{m}| = tW_T(S) $, which is a restatement of the definition of $W_T(S)$, we get
                \begin{align*}
                	\sum_{m \in T} q_{m+1} \geq \sum_{m \in T}\frac{2|Q_{m+1}|}{n^2}\geq \frac{tW_T(S)|S|}{n^2} - O\left(\frac{t^2}{n^2}\right).
                \end{align*}
                By the claim assumption and the definition of $t$, this is at least $50n - O(\log^2 n) \ge 40n$, and the claim follows.
                \end{proof}
By Claim~\ref{claim:ProbExp}, if both $\mathcal{A}$ and $\mathcal{B}$ occur, then 
$$\left|\sum_{m \in T}(X_{m+1} - q_{m+1})\right| =  \sum_{m \in T}q_{m+1} \ge 40n.$$
So by (\ref{freed1}), $\pr{\mathcal{A} \land \mathcal{B}} \le e^{-2n}$ as required. Observing that $\mathcal{B}$ holds if $S$ is independent in $G_{m_1 + t}$ and taking the union bound over $S$, the probability that (i) does not hold is at most $2^ne^{-n}$. This completes the proof of \ref{itm:covercontrol1}. 
        
        The proof of \ref{itm:covercontrol2} follows very similarly. Let $U$ be a vertex set of order $\frac n2$. Now, for $m \geq m_1$, let $R_{m+1}$ be the set of vertex pairs intersecting $U \setminus D_{m}$ and let $r_{m+1} := \frac{|R_{m+1}|}{\binom{n}{2}-m}$. If $W_T(U) \geq \frac n2 - \frac{an^2}{t}$, then
        \begin{align*}
            \sum_{m \in T} r_{m+1} \leq \sum_{m \in T} \frac{4|R_{m+1}|}{n^2}
            \leq \sum_{m \in T} \frac{4n \left(|U| - |U \cap D_{m}| \right)}{n^2}
            = \frac {4t}{n}\left(\frac n2 - W_T(U) \right) \leq 4an.
        \end{align*}
        
        Let $Z_m$ be the indicator random variable of the event $e_m \in R_{m}$. Then $\prm{m}{Z_{m+1} = 1 } = r_{m+1}$.
        Note that the edge $e_m$, which is incident to $U$, can be rejected only if $e_m \in R_{m}$. Therefore $\sum_{m \in T} Z_{m+1}$ counts the number of rejected edges incident to $U$.
        Analogously to part \ref{itm:covercontrol1}, applying Corollary \ref{cor:freedman} to $r_m$ with the constant $L = a$ gives
        $$\pr{\left(\sum_{m \in T} Z_{m+1} \geq 5an \right) \land \left(W_T(U) \geq \frac n2 - \frac{an^2}{t} \right)} \leq e^{-an/10}.$$
        Taking the union bound over all possible $U$ and recalling the hypothesis $a \geq 10$, we get that \ref{itm:covercontrol2} holds with probability at least $1-2^ne^{-n} = 1-o(1)$.
        
    \end{proof}

     We now apply Lemma~\ref{lem:covercontrol} to control rejected edges adjacent to an optimal cover and use this to prove Lemma~\ref{lem:rejedge}. As discussed above, this immediately implies Theorem~\ref{thm:finalgraph}\eqref{itm:me}. 
		\begin{lem}\label{lem:rejedge}
            Let  $m:= 1.1 n \log n$. With high probability, if $C$ is any optimal cover in $\Gour{m}$, then $O(n)$ edges of $\Gall{m}$ incident to $C$ have been rejected. 
        \end{lem}
            \begin{proof}
            	Set $t = m- \gamma_0 n$  and $T:= T(\gamma_0 n,t)$, where $\gamma_0$ is a large constant chosen so that the conclusion of Lemma~\ref{lem:betterpm} holds. Namely, with probability $1- O\left(n^{-2} \right)$, 
            	\begin{equation} \label{eq:taum}
            	\nu\left(\Gour{\ell} \right) \geq \frac n2 \left(1 - e^{3\ell/(10n)} \right) \text{\quad for \quad} \ell \in T.
            	\end{equation}
                We have $|D_\ell| \ge \nu\left(\Gour{\ell} \right)$. So~\eqref{eq:taum} implies
            	    \begin{align}\label{eq:weight}
            	        W_T(V) \geq \frac n2 - \frac {n}{2t} \sum_{\ell \in T} e^{-3\ell / (10n)} 
            	        \geq \frac n2 - \frac{ne^{-3\gamma_0 n/(10n)}}{2t\left(1-e^{-3/(10n) } \right)}
            	             \geq \frac n2 - \frac {20n^2}{t} ,
            	    \end{align}
       where we used that $1-e^{-x} \geq x/2$ for $x\leq 1/2$.     	    
       As $V \setminus C$ is an independent set in $\Gour{m}$,  Lemma~\ref{lem:covercontrol} \ref{itm:covercontrol1} implies that with high probability $W_T(V \setminus C) \leq \frac{150 n}{\log n}$. Combining this with (\ref{eq:weight}) gives that $W_T(C) \geq  \frac n2- \frac{200n}{\log n}$ with high probability. Now applying Lemma \ref{lem:covercontrol} \ref{itm:covercontrol2} to $C$ gives that, with high probability, the number of edges incident to $C$ rejected during $T$ is $O(n)$. 
         
         Clearly at most $\gamma_0 n$ vertex pairs are rejected before time $\gamma_0 n$.  
         This completes the proof of the lemma and of Theorem~\ref{thm:finalgraph}(iii). 
          \end{proof}

\section{Rigidity and uniqueness of an optimal cover}		\label{section:cover}
We now turn our attention to analysing optimal covers throughout our process. In this section we will first prove Lemma~\ref{lem:coverfrozen}, which concerns the `rigidity' of an optimal cover, and Theorem~\ref{thm:finalgraph}\eqref{itm:ovc}, which tells us that with high probability, $\Gour{4n \log n}$ has a unique optimal cover.

We start with an elementary observation. For $m \ge 0$, conditioned on $(e_i)_{i \le m}$, let $p_{m+1}$ be the probability that $e_{m+1}$ is acceptable for $G_m$. Recall that $e_{m+1}$ is acceptable
whenever it is incident to an optimal vertex cover. So for $0 \le m \le 2n \log n$, if $\nu(G_m)=\tau(G_m)$ is of order at least $n\left(\frac{1}{2} - r(n)\right)$, then  
\begin{equation} \label{qmbound}
    	p_{m+1} \geq \frac{\binom {n}{2} - \binom {\left(\frac{1}{2} + r \right)n}{2}- m}{\binom {n}{2}-m} \geq \frac{n^2 - n - \left(\frac 12+r \right)^2n^2-4n \log n}{n^2}\geq \frac{3}{4} - 2r- \frac{5\log n}{n}.
\end{equation}
 
 A key ingredient for the proof of Lemma~\ref{lem:coverfrozen}  is the following lower bound on the number of edges accepted into our graph during a certain time period.  
     
\begin{lem} \label{lem:lowerboundedges}
For $\gamma >0$, let $T:=T(m_0, \gamma n \log n)$. Let $G$ be the graph consisting of all edges accepted into $(\Gour{m})$ during the period $T$. With high probability,
	\begin{equation} \label{eq:lowerboundedges}
		|E(G)| \geq  \sum_{m \in T} p_{m} - n\sqrt{\log n}.
	\end{equation}
\end{lem}
\begin{proof}

	For $m \ge 0$, define $X_m$ to be the indicator random variable of the event that $e_m$ is accepted and $Y_m := X_m - p_m$. By definition of $p_m$, we have $\erm{m-1}{Y_m}=0$, so $\left(\sum_{i = m_0}^{m'} Y_i \right)$ is a martingale. Set  $Y := \sum_{m \in T} Y_m$. Moreover, $|Y_m| \leq 1$ for each $m$, so we can apply Azuma's inequality (Theorem \ref{thm:azuma}) with $\lambda = n \sqrt{ \log n}$.  We get
    	$$\pr{Y < - n \sqrt{\log n}} \leq \exp \left(-\frac{n^2 \log n}{2 \gamma n \log n} \right).$$
    It follows that, with probability $1-e^{-\Omega(n)}$,
    	$$|E(G)| = \sum_{m \in T}X_m \geq \sum_{m \in T} p_{m} - n\sqrt{\log n}.$$
\end{proof}

Call a time step $m$ $\phie$-\emph{flexible} if $|D_m| \geq \left(\frac 12 + \phie \right)n$. In all our arguments, this notion will be used with the same function $\phie = \phie(n) =(\log n)^{-1/10} $.
We now prove Lemma~\ref{lem:coverfrozen}. Our proof relies on the statistics of the process -- informally, if many time steps were $\phie$-flexible, we would be accepting too many edges.

	\begin{proof}[Proof of Lemma~\ref{lem:coverfrozen}]
		Recall that  $m_0 = n \sqrt{\log n}$,  $T= T(m_0,\gamma n \log n)$ and $\phie = (\log n)^{-1/10}$.
		For ease of notation, let $t:= |T| = \gamma n \log n$. Let us assume, in order to obtain a contradiction, that the number of $\phie$-flexible steps is greater than $\phie t$. Let $G$ be the graph consisting of edges accepted during the interval $T$, and $H := \Gall{m_0+t}$. Note that $H$ is distributed as $G_{n, m_0+t}$.
		
		By Lemma~\ref{lem:betterpm}, we know that with high probability $\nu(\Gour{m_0})$ contains a matching of size $\frac{n}{2}(1 - r)$, where $r:= r(n) = e^{-3m_0/(10n)} \ll \log^{-1} n \ll \phie$. As this property is increasing,  we have that $\Gour{m}$ with $m \ge m_0$ also contains a matching of at least this size. Thus, using (\ref{qmbound}), with high probability for all $m \ge m_0$, we have $p_{m} \ge \frac{3}{4} - 3r$. Moreover, if the step $m$ is flexible, we have a stronger bound $p_{m} \geq \frac 34 + \frac{\phie}{2}$. 
		
		By applying Lemma~\ref{lem:lowerboundedges} and the above analysis, with high probability we have 
			\begin{equation}
			\label{eq:upeT}
			|E(G)| + n \sqrt{\log n} \ge \sum_{m \in T} p_{m} \geq  \frac {3t}{4} + \frac {\phie}{2} \cdot \phie t  - 3r(1 - \phie) t  \geq \left(\frac 34 +\frac{\phie^2}{4}\right) t.
			\end{equation}
		Let $m_1:= m_0 + t=t(1+O\big((\log n)^{-1/2})\big)$.  Let $C$ be an optimal vertex cover in $G_{m_1}$, which has size at least $\frac{n}{2}(1 -\log^{-1} n)$. As $V \setminus C$ is an independent set of size at most
		$\frac{n}{2}(1 +\log^{-1} n)$, by applying Theorem~\ref{thm:edgedistribution} to graph $H \sim G_{n, m_1}$ we obtain 
			\begin{align*}
            	|E(\Gour{m_1})| \leq |E_H(C)| + |E_H(V \setminus C,C)| &\leq \left(\frac 34 + O\left((\log n)^{-1/2}\right) \right) m_1 \\
                &= \left(\frac 34 + O\left((\log n)^{-1/2}\right) \right) t.
               \end{align*}
		As $\Gour{m_1}\supseteq G$, this contradicts (\ref{eq:upeT}). This completes the proof of Lemma~\ref{lem:coverfrozen}.
	\end{proof} 

We finish the section by proving Theorem~\ref{thm:finalgraph}\eqref{itm:ovc}, which will be restated here to aid the reader. The key idea is to analyse the set of rejected pairs and to use these to obtain information about an optimal cover. 

\begin{thmii}
\emph{With high probability $G_{4n\log n}$ has a unique optimal vertex cover. }
\end{thmii}

\begin{proof}
Set $m_0:= \frac{5}{4} n\log n$, $m_1:= \frac{3}{2} n\log n$ and $m_2:= 4 n \log n$. By Theorem~\ref{thm:finalgraph}\eqref{itm:pm}, with high probability $G_{m_0}$ has a perfect matching and hence an optimal cover of cardinality $n/2$. For each $i \ge 0$, let $\mathcal{C}_i$ be the set of optimal covers of $\Gour{i}$. Observe that for all $i \ge m_0$, we have $\mathcal{C}_{i+1} \subseteq \mathcal{C}_i$, as adding edges can only eliminate optimal covers. 

Let $E_0:= \{e_1,\ldots,e_{m_0}\}$ and $E_1:= \{e_{m_0 + 1},\ldots, e_{m_1}\}$. Note that $G':=\Gall{m_1} \setminus E_0 \sim G_{n,m'}$, where $m' = \frac{n}{4}\log n$. So by Theorem~\ref{thm:edgedistribution}, with high probability, in $G_{n,m'}$ every set of cardinality $n/2$ contains at least $(1 - O((\log n)^{-1/2}))\frac{n}{16}\log n$ edges. 
However, $\Gour{m_1}$ contains an independent set of cardinality at least $n/2$. So, with high probability, at least $(1 - O((\log n)^{-1/2}))\frac{n}{16}\log n$ pairs of $E_1$ are rejected. Let $U$ be the set of vertices spanning these rejected pairs.

We will next show that $|U| \ge (1 - (\log n)^{-1/4})\frac{n}{2}$. Suppose, for a contradiction, that $|U| < (1 - (\log n)^{-1/4})\frac{n}{2}$. Then applying Theorem~\ref{thm:edgedistribution} shows that in 
$G' \sim G_{n,m'}, m'= \frac{n}{4}\log n$, with high probability we have 
$$|E(U)| \le \frac{1}{16}n\log n - \Omega(n (\log n)^{3/4}) < (1 - O((\log n)^{-1/2}))\frac{1}{16}n\log n,$$
a contradiction for $n$ sufficiently large.

Now observe that if $(u,v) \in E_1$ is rejected, then there exists no $C \in \mathcal{C}_i$ such that $u \in C$ or $v \in C$. This together with our observation from the first paragraph shows that, in fact, every cover in $\mathcal {C}_{m_1}$ contains neither $u$ nor $v$. So in particular, no vertex of $U$ is contained in a cover in $\mathcal{C}_{m_1}$ (or $\mathcal{C}_{m_2}$). 

Now let $E_2:= \{e_{m_1 +1},\ldots, e_{m_2}\}$ and consider $e := (u,v) \in E_2$ such that $u \in U$ and $v \notin U$. As in the previous paragraph, if $e$ is rejected then no cover in $\mathcal{C}_{m_2}$ contains $u$ or $v$. However, if $e$ is accepted, since (by the previous paragraph) $u$ is in no cover of $\mathcal{C}_{m_2}$, then every cover in $\mathcal{C}_{m_2}$ contains $v$. So for each $v \in V\setminus U$, let $\mathcal{E}_v$ be the event that $E_2$ contains a pair $(u,v)$ for some $u \in U$. If $\mathcal{E}_v$ occurs, then we know that either $v$ is contained in every cover or no cover in $\mathcal{C}_{m_2}$. 

So if, $\mathcal{A}:= \bigcup_{v \in V\setminus U}\mathcal{E}_v$ occurs, then $\mathcal{C}_{m_2}$ contains a unique optimal cover. It remains to show that $\mathcal{A}$ occurs with high probability. For a particular $v \in V$, say that $e_i$ is \emph{good} for $v$ if $e_i = (u,v)$ for some $u \in U$. Let $d_i(v)$ be the degree of $v$ in $\Gall{i}$. For $i \le m_2$, the probability that $e_i$ is good for $v$ is at least
$$\frac{|U| - d_i(v)}{\binom{n}{2} - i +1} \ge  \frac{(1 - o(1))}{n},$$
as $|U| \ge (1 - (\log n)^{-1/4})n/2$ and, by Claim~\ref{maxdeg} we have $d_i(v) = O(\log n)$, as $\Gall{i} \subseteq \Gall{10n \log n}$. 
So the number of pairs in $E_2$ that are good for $v$ is at least $X$, where $X \sim \Bin \left(\frac{5}{2}n \log n, \frac{1 - o(1)}{n} \right)$. We have
$$\pr{X=0} = \left(1 - \frac{1 - o(1)}{n}\right)^{\frac{5}{2}n \log n} \le e^{-2 \log n}.$$
Hence the probability that $\mathcal{E}_v$ does not occur is at most $\frac{1}{n^2}$. Applying the union bound over the vertices in $V \setminus U$ gives that $\mathcal{A}$ occurs with high probability, as required. 
\end{proof}

\section{Delayed perfect matching threshold} \label{section:delay}
The focus of this section is to prove Theorem~\ref{thm:delayedpm}, which says that for  $m = \left(\frac 12 + \frac{1}{70} \right)n \log n$,  $\Gour{m}$ typically has isolated vertices. Throughout this section, we set $m_2:= \frac{3}{8}n \log n$ and $m_3:= 2n \log n$. We assume that $\Gour{m_2}$ has a subgraph $H$
 which has a perfect matching and satisfies
\begin{equation} \label{eq:num2}
	|V(H)| \geq  n - ne^{-\frac{3m_2}{10n}} \geq  n\left(1 - {n^{- \frac{1}{20}}}\right) 
    \text{\quad and \quad} \delta(H) \geq 10^{-3} \log n, \end{equation}  
    as by Lemma~\ref{lem:precisedegrees} and Corollary~\ref{cor:mindegreepm} this event occurs with high probability.

In order to show that our graph does not contain a perfect matching, we will carefully track the number of vertices with at most one neighbour in $G_m$. In order to understand how these vertices are used at a particular time step to extend a current maximum matching, we require some information on separation of small-degree vertices in $\Gour{m}$ (see Lemma~\ref{lem:agreeable} below). One important consequence will be that, typically in $G_m$, no vertex has three neighbours of degree one. For $u,v \in V$, let $\dist_m(u,v)$ be the distance from $u$ to $v$ in $\Gour{m}$ and let $\dist^{\text{all}}_m(u,v)$  be the distance from $u$ to $v$ in $\Gall{m}$. We take the convention $\dist (u, u) = 0$ for any underlying graph and any vertex $u$. The following lemma will mostly be applied with $a = 3$.

\begin{lem}\label{lem:agreeable}
For sufficiently small $\epsilon>0$, there exists $\delta>0$ such that with high probability the following statement holds. For a positive integer $a\leq 5$ and all $(1+\epsilon) n \log n / a \le m \le m_3$, there are no distinct vertices $u_1, \dots, u_a$ such that, for all $i,j \in [a]$, we have $\deg_{\Gour{m}}(u_i) < \delta \log n$  and $\dist_{m}(u_i,u_j) \le 10$. 
\end{lem}

\begin{proof}
Let $\ell = (1+\epsilon) n \log n / a$, and let  $\mathcal{A}$ be the event that there exist $u_1,\dots, u_a$ such that for all $i,j \in [a]$, we have $\deg_{\Gour{\ell}}(u_i) \leq \delta \log n$ and $\dist^{\text{all}}_{m_3}(u_i,u_j) \le 10$. Since for any $\ell\le m \le m_3$ and any $i,j \in [a]$, we have $\deg_{\Gour{m}}(u_i) \ge \deg_{\Gour{\ell}}(u_i)$ and $\dist_m(u_i,u_j) \le \dist_{m_3}(u_i,u_j) \le  \dist^{\text{all}}_{m_3}(u_i,u_j)$, it suffices to show that $\pr{\mathcal{A}} = o(1)$. 

  Let $U := \left\{ u_1, \dots, u_{ a } \right\}\subseteq V$ and define
   $$D := \sum_{i=1}^{ a } \deg_{\Gour{\ell}} (u_i).$$ If $\dist^{\text{all}}_{m_3}(u_i,u_j) \le 10$ for all $i,j \in [a]$, then there is a connected subgraph of $\Gall{m_3}$ on at most $11\cdot(a-1) <50$ vertices containing $U$. In particular, this subgraph has a spanning tree.
    Let $\mathcal{T}_F$ be the event that $\Gall{m_3}$ contains a fixed labelled tree $F$ with $U \subseteq V(F)$ and $|V(F)| := f < 50$.
    
   First we prove an upper bound on  $\pr{D < \three \delta \log n \given \mathcal{T}_{F}}$. Let $T':= \{i: e_i \in E(F)\}$ and $T:= T(\gamma n, \ell - \gamma n)$. Let $Y_m$ be the indicator random variable of the event that $e_{m}$ is an acceptable pair with one endpoint in $U$ and the other in $V\setminus F$. Using Corollary~\ref{cor:linmatch} \ref{itm:linmatch2}, for $\gamma$ sufficiently large and any $\gamma n +1 \le m \le \ell$ such that $m \notin T'$ we have
   that there are at least $a(1/2 - \epsilon/4 )n$ acceptable pairs touching one of the vertices in $U$. Therefore
		$$\prm{m-1}{Y_m = 1} \ge \frac{a(1/2 -\epsilon/4)n}{(1+o(1)){n \choose 2}} \geq a \cdot \frac{1-\epsilon/2}{n}.$$
	Thus, letting  $X \sim \Bin \left(\ell(1-o(1)), \frac{a(1- \epsilon/2)}{n}\right)$, we have
      $$\pr{\sum_{m \in T \setminus T'}{Y_m} <k \given \mathcal{T}_{F}} \leq \pr{X < k}.$$
Noting that $\er{X}  =  \frac{ a(1-\epsilon/2)}{n}\cdot \ell (1-o(1)) > \left(1 + \frac{\epsilon}{4} \right) \cdot \log n$ and applying  Lemma \ref{lem:EltCalc} with $\delta$ sufficiently small and $\eta=\frac{\epsilon}{8}$, we get
    $$\pr{D < \three \delta \log n \given \mathcal{T}_F}  \leq \pr{\sum_{m \leq \ell} Y_m \leq \delta \cdot \three \log n \given \mathcal{T}_F}\leq e^{-\left( 1 +\frac{\epsilon}{10}\right)  \log n }.$$
 
  For a fixed tree $F$ we now show that $\pr{\mathcal{T}_F} = O \left(\left(\frac{\log n}{n} \right)^{f-1}\right)$. To see this, note that the probability that the random graph $\Gnp$ with $p = m_3 \binom{n}{2}^{-1}$ contains $F$ is precisely $p^{f-1}$. Using Lemma \ref{lem:gnpisfine}, the fact  that $\mathcal{T}_F$ is a monotone increasing event and  the crude estimate $p < \frac{5 \log n}{n}$, we get   
     $$\pr{\mathcal{T}_F} \leq 3\left( \frac{5 \log n}{n} \right)^{f-1} = O \left(\left(\frac{\log n}{n} \right)^{f-1}\right).$$
     Putting this together with the previous bound gives
     $$\pr{(D < \three \delta \log n) \land \mathcal{T}_F} =  n^{- 1- \frac{\epsilon}{10}} \cdot O \left(\left(\frac{\log n}{n} \right)^{f-1}\right) = o\left(n^{-f} \right).$$
      As there are $O \left(n^{50}\right)$ choices for $F$ and $U \subseteq V(F)$, we can take the union bound over $U$ and $F$ to get  $\pr{\mathcal{A}} \leq \pr{D < \three \delta \log n}  = o(1).$  
\end{proof}

Given this lemma, and noting that $m_2 = \frac 38 n \log n>(1+ \epsilon) n \log n /3$ for a positive constant $\epsilon$, we may assume for the rest of the proof that whenever $m \ge m_2$, no vertex in $\Gour{m}$ has three neighbours of degree one. We are now able to introduce the specific definitions we need for the main part of our proof.

Say that a vertex is an \emph{isolate} or \emph{is isolated} in $\Gour{m}$ if it has no neighbours in $\Gour{m}$. If there exist two vertices $v_1,v_2$ of degree one in $\Gour{m}$ that share a neighbour, then we arbitrarily choose one to be a \emph{quasi-isolate} and the other to be its \emph{partner} (by the previous paragraph there cannot typically be a third vertex of degree one sharing a neighbour with $v_1$ and $v_2$). The partner is \emph{not} a quasi-isolate. Let $J_m$ denote the set of vertices that are either isolates or quasi-isolates in $\Gour{m}$. 

Define $M_m$ to be a maximum matching in $\Gour{m}$ with $V(M_m) := C_m \cup B_m$, where $C_m$ is an optimal vertex cover in $\Gour{m}$, chosen subject to the following: 
	\begin{itemize}
		\item[(M1)] If $v \in C_m$ is matched to $u \in B_m$ by $M_m$ and has a neighbour $w\not\in V(M_m)$, then  $\deg_{G_m}(w) \leq \deg_{G_m}(u)$.
		\item[(M2)] $M_m$ contains no quasi-isolate.
	\end{itemize} 
Achieving (M1) is possible by swapping $u$ for $w$ if this is not the case (this will give another matching of the same size). Similarly, by swapping a quasi-isolate with its partner, (M2) is achievable. So such a matching $M_m$ always exists. Note that there may be many choices for such an $M_m$. 

By definition of $M_m$, the set $J_m$ is disjoint from $V(M_m)$. Define the set of \emph{helpers} to be $H_m  := V \setminus (V(M_m) \cup J_m)$ and observe that the number of helpers $|H_m|$ is independent of the particular choice of $M_m$. Intuitively, the helpers are vertices attached to $M_m$ so as to create many vertex pairs which would extend $M_m$.\\

\begin{figure}[htbp]
\centering
\includegraphics[width=0.5\textwidth]{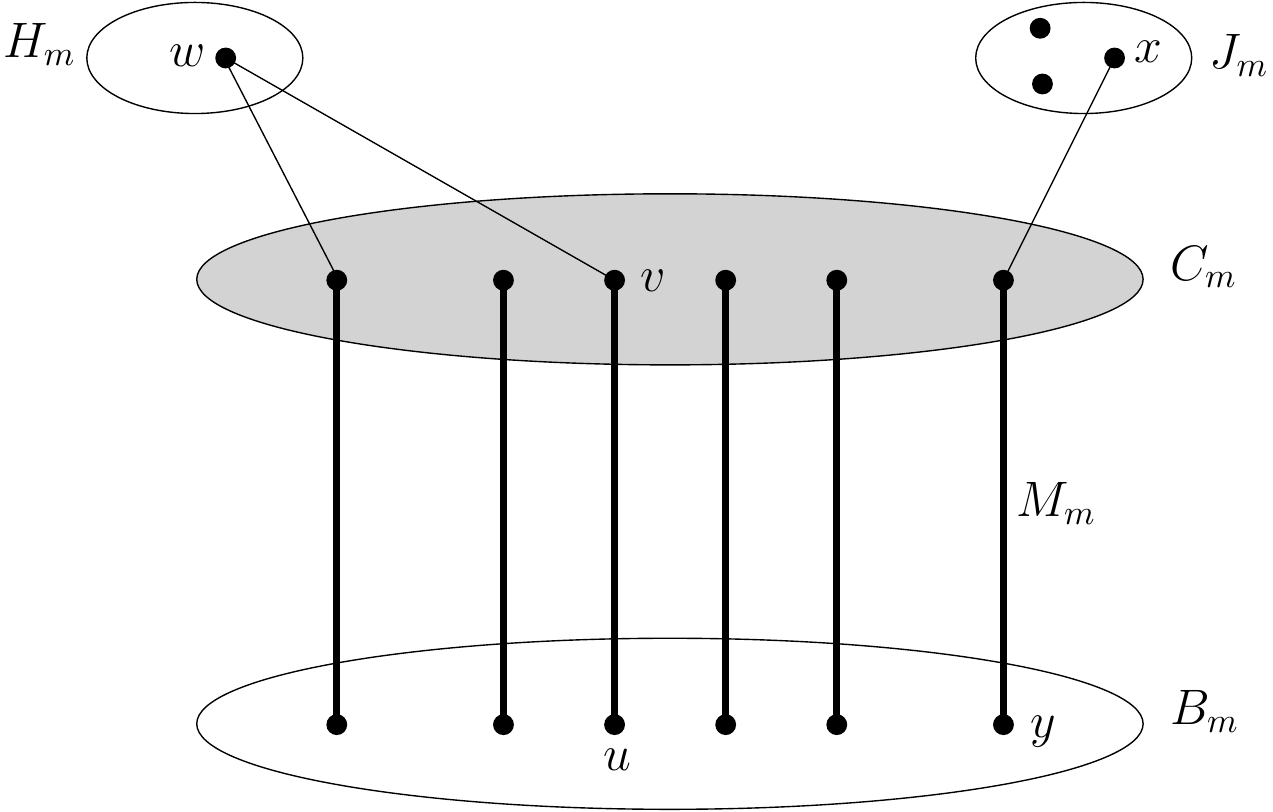}
\caption{An illustration of helpers, isolates, quasi-isolates and the matching $M_m$. The vertex $x$ is a quasi-isolate and $y$ is its partner, so both have degree 1 in $G_m$. By (M1), $\deg(w) \le \deg(u)$. Recall that all additional edges are incident to $C_m$.}\label{fig:helpers}
\end{figure}

Our aim is to show that typically $\Gour{m}$ contains no helpers for a constant proportion of time steps $m$ (Lemma~\ref{lem:zerohelpers}). In other words, there is  a matching covering all vertices apart from isolates and quasi-isolates. At such steps, the rate of losing isolated vertices is slower than in $\Gall{m}$, so Theorem \ref{thm:delayedpm} will follow. The argument consists of two  main lemmas. The first one is  an analogue of the hitting time result of Bollob\'as and Thomason \cite{bt83} for our setting.


\begin{lem} \label{lem:notwohelpers}
With high probability, the number of steps $m \ge m_2$ at which $|H_m|\geq 2$ is $O\left( n^{1-1/20} \right)$.
\end{lem}

The second lemma tells us that typically at many time steps we have a non-zero even number of isolates or quasi-isolates.\\

\begin{lem} \label{lem:parity}
With high probability, for at least $\frac{n \log n}{33} $ steps $m \ge m_2$ we have $|J_m| > 0$ and $|J_m|$ is even. 
\end{lem}

The main work in this section is devoted to proving these two lemmas. Lemma~\ref{lem:notwohelpers} will be proved in Section~\ref{subsec:zerohelpers}, and Lemma~\ref{lem:parity} will be proved in Subsection~\ref{subsec:par}. Before doing this, let us first show how the two lemmas imply the desired result.

\begin{lem}\label{lem:zerohelpers}
With high probability, for at least $\frac{ n \log n}{34}$ time steps $m \geq m_2$, we have $H_m = \emptyset$ and $J_m \not= \emptyset$.
\end{lem}

\begin{proof}
 Since the number of vertices $n$ is even, $|J_m|+|H_m| = n - |V(M_m)|$ is always even. In other words, $|J_m|$ and $|H_m|$ have the same parity, so Lemma~\ref{lem:parity} gives that with high probability, $|H_m|$ is even and $J_m \not= \emptyset$  for at least $\frac{n \log n}{33} $ steps $m$. However, Lemma~\ref{lem:notwohelpers} shows that, with high probability, $|H_m| \ge 2$ for $o(n \log n)$ steps. The result follows.  
 \end{proof}
 
\subsection{Extending maximum matchings in $G_m$}  \label{subsec:zerohelpers}
Our goal in this subsection is to prove Lemma~\ref{lem:notwohelpers}. That is, we wish to show that typically, \emph{most} of the time $\Gour{m}$ contains a matching which is as large as possible given the restrictions posed by isolates and quasi-isolates. Therefore it should not come as a surprise that we need to take a careful look at the structure of $\Gour{m}$.

The aim is to show that, for $m_2 \le m \le m_3$, if we have at least two helpers in $\Gour{m}$, then there are many choices of $e_{m+1}$ whose addition would increase the size of a maximum matching (see Lemma~\ref{lem:extendmatching} for the precise statement). Firstly, we need an expansion property of our graph, stronger than that given by  Lemma~\ref{lem:precisedegrees} \eqref{itm:precisedegrees2}. 
 For the rest of the section, fix $\delta_0$ to be the constant from Lemma~\ref{lem:agreeable} applied with $\epsilon = 0.01$ (so that $m_2 > (1+\epsilon) n \log n / 3$).

\begin{lem}\label{lem:superexpansion}
With high probability the following statement holds. There exists $\zeta >0$ such that for all $m_2 \le m \le m_3$ and for every set $S \subseteq V$ of cardinality at most $ \frac{n}{4\log n}$  where every vertex in $S$ has at least $\delta_0 \log n$ neighbours outside of $S$, we have $|N_{\Gour{m}}(S)| \ge \zeta \log n |S|$. 
\end{lem}
\begin{proof}
By Theorem~\ref{thm:edgedistribution}, with high probability every set $S$ with $|S| \le \frac{n}{4\log n}$ satisfies
$$E_{\Gour{m}}(|S|, |N_{\Gour{m}}(S)|) \le E_{\Gour{m_3}}(|S|, |N_{\Gour{m_3}}(S)|) \le \frac{2 m_3 |S||N_{\Gour{m_3}}(S)|}{n^2} + \lambda \sqrt{\frac{m_3|S||N_{\Gour{m_3}}(S)|}{n}},$$
where $\lambda$ is a positive constant granted by the Theorem. If for some $S$ we have $|N_{\Gour{m_3}}(S)| < \zeta \log n |S|$ with a positive constant $\zeta$, then the right-hand side is at most 
$$ \frac{4\zeta\log^2n|S|^2}{n} + \lambda|S|\log n\sqrt{2\zeta}\le |S| \log n( \zeta + \lambda\sqrt{2 \zeta}) < \frac{\delta_0}{2} |S| \log n,$$
when $\zeta$ is chosen to be sufficiently small. But by hypothesis, $E_{\Gour{m}}(|S|, |N_{\Gour{m}}(S)|) \ge \delta_0 |S| \log n$, a contradiction. 
\end{proof}

For the rest of the section we will assume that our process throughout steps $m_2 \le m \le m_3$ satisfies the properties granted by Lemma \ref{lem:agreeable} and Lemma~\ref{lem:superexpansion}, without explicitly referring to the probabilistic statements. 

 Say that a vertex is \emph{large} if it has degree at least $\delta_0 \log n$ in $G_m$ and \emph{small} otherwise. For a vertex $u \in V(M_m)$, let $g(u)$ denote the vertex that is matched to $u$ in $M_m$. Similarly for a set $S \subseteq V(M_m)$, let $g(S):= \{g(u):= u \in S\}$. For $m \ge m_2$ we obtain the following properties as an easy consequence of Lemma~\ref{lem:superexpansion} and the definition of $M_m$. 
	\begin{itemize}
			\item[(P1)] For any $v$, $N^5_{G_m}[v]$ contains at most two vertices of degree at most $\delta_0 \log n$ in $G_m$, where $N^i[v]$ denotes the set of vertices within distance at most $i$ from $v$. 
			\item[(P2)] If a helper has degree one, then its neighbour $w$ in $C_m$ satisfies $\deg_{G_m}(g(w)) \ge 2$. 
		\end{itemize}

We will now introduce some more definitions. Let $M$ be a matching in $G_m$. Say that a path $P:= u_1\ldots u_k \in G_m$ is $M$-\emph{alternating} if the edges of $P$ alternate between being in and out of $M$. Say that an $M$-alternating path $P$ \emph{augments} $M$ if $u_1, u_k \notin V(M)$. Observe that if $P$ augments $M$, then $M':= (M \setminus P) \cup (P \setminus M)$ is a larger matching in $G_m$ than $M$. Our next lemma will show that if $|H_m| \ge 2$, then there are $\Omega(n^2)$ vertex pairs that have not yet been offered that will create a path that augments a maximum matching in $G_m$.

\begin{lem} \label{lem:extendmatching}
Let $m_2 \leq m < m_3$. If $|H_m|\geq 2$, then for some absolute constant $\eta > 0$, the probability that $\nu(G_{m+1}) > \nu(G_m)$ is at least $\eta$.
\end{lem}
\begin{proof} 
Let $u_1,v_1$ be distinct vertices in $H_m$. For ease of notation, let $C = C_m$, $B= B_m$ and $M:= M_m$. We will find $\eta n^2$ distinct vertex pairs $(x,y)$ such that $x,y \in B$ (and hence $xy \notin E(G_m))$, and such that $G_m \cup \{xy\}$ contains an $M$-alternating path from $u_1$ to $v_1$ that augments $M$. As $m = o(n^2)$, at least $\frac{\eta n^2}{2}$ of these pairs are not contained in $e_1,\ldots,e_m$ and hence $\pr{\nu(G_{m+1}) > \nu(G_m)} \ge \eta$. 

In order to find the pairs $(x,y)$, whose addition creates a path in $G_m$ that augments $M$, we will first find disjoint paths $P_1$ from $u_1$ to $u^*$ and $Q_1$ from $v_1$ to $v^*$, where $u^*$ and $v^*$ are large vertices in $B_m$. We will then be able to use the expansion properties of the graph induced by the large vertices to extend $P_1$ and $Q_1$ to many disjoint $M$-alternating paths terminating at distinct pairs of vertices in $B$. The main difficulty in the proof is finding the large family of extensions.

We begin by finding $P_1$ and $Q_1$.
\begin{claim}
There exist disjoint $M$-alternating paths $P_1 = u_1 \ldots u^*$ and $Q_1 = v_1 \ldots v^*$, where $u^*$ and $v^*$ are large vertices and $|P_1|,|Q_1| \le 5$.
\end{claim}
\begin{proof}
If $u_1$ and $v_1$ share a neighbour then they do not both have degree one, otherwise $u_1$ or $v_1$ would be a quasi-isolate. So we are able to pick $u_2 \not= v_2 \in C$ such that $u_1u_2$ and $v_1v_2$ are edges of $G_m$. If $u_3:= g(u_2)$ and $v_3:=g(u_3)$ are both large, set $P_1:=u_1u_2u_3$ and $Q_1:=v_1 v_2 v_3$. 

Therefore we may assume that $u_3$ is small. It follows from (M2) that $u_1$ is also small as $\deg_{G_m}(u_1) \leq \deg_{G_m}(u_3)$. By  (P2), $\deg(u_3) \ge 2$ and so pick $u_4 \in C$ to be a neighbour of $u_3$ distinct from $u_2$. Define $u_5:= g(u_4)$. By Lemma~\ref{lem:agreeable}, $N^2_{G_m}[u_2]$ contains at most two small vertices, so $u_5$ is large. Set $P_1:= u_1\ldots u_5$.

We are left with two cases for $v_3$. 
If $v_3$ is small, then $N^3_{G_m}[u_2] \cap N^3_{G_m}[v_2] = \emptyset$ since otherwise we could find small vertices $u_1$, $u_3$, $v_3$ (and even $v_1$) at mutual distances at most 10. Therefore we can find $v_5$ analogously to $u_5$ and set $Q_1:= v_1\ldots v_5$. If $v_3$ is large and $u_4 \neq v_2$, we may set $Q_1 = v_1 v_2 v_3$. Finally, let $v_3$ be large and $u_4 = v_2$. Then $v_1$ is large since otherwise $N^{3}_{G_m}[u_4]$ contains three small vertices. Hence $v_1$ has another neighbour $v_2' \notin \{ u_2, v_2 \}$, so  set $v_3':=g(v_2')$. The vertex $v_3'$ is also contained in $N^{3}_{G_m}[u_4]$ (recalling that $u_4 = v_2$)
and therefore it is a large vertex (otherwise we again will have three small vertices in $N^{3}_{G_m}[u_4]$).
Thus we may set $Q_1 = v_1 v_2'  v_3'$. 
\end{proof}

The following claim gives a large family $\mathcal{P}$ of pairs of $M$-alternating paths extending $P_1$ and $Q_1$, such that  each pair $(P, Q) \in \mathcal{P}$ yields a unique pair of endpoints $(u_P, v_Q)$.  Let $c:= \frac 14 \min\{\delta_0, \zeta\}$, where $\zeta$ is the constant from Lemma~\ref{lem:superexpansion}.

\begin{claim}\label{cl:manypaths}
There exists a family $\mathcal{P}$ of pairs of paths $(P,Q)$ with the following properties. 
\begin{enumerate}[(i)]
\item\label{itm-disjoint} $P = P_1P_2$ and $Q = Q_1Q_2$ are $M$-alternating disjoint paths, where $P_2 = u^*\ldots u_P$ and $Q_2 = v^*\ldots v_Q$ with $u_P,u_Q \in B$. 
\item\label{itm-big} $|\{(u_P,v_Q): (P,Q) \in \mathcal{P}\}| \ge \frac{c^2n^2}{16}$.
\end{enumerate} 
\end{claim}

Let us first explain why the lemma follows from this claim, before proving it. By \eqref{itm-disjoint}, for each $(P,Q) \in \mathcal{P}$, the addition of the edge $u_Pv_Q$ creates an $M$-alternating path that augments $M$. By \eqref{itm-big}, we have at least $\frac{c^2n^2}{16}$ such vertex pairs. At most $m = O(n \log n)$ of these pairs appear in $e_1,\ldots, e_m$, which implies the statement of the lemma with $\eta := \frac{c^2}{16}$.

\begin{proof}[Proof of Claim~\ref{cl:manypaths}]

We will construct disjoint sets $U_1,\ldots, U_k, V_1, \ldots, V_k \subseteq B\setminus (P_1 \cup Q_1)$. These sets will have the \emph{path property} that, for every $j$ and every pair $(x_j,y_j) \in U_j \times V_j$, there exist sequences $x_1,\ldots, x_{j-1}$ and $y_1,\ldots,y_{j-1}$, such that for all $i < j$ we have $x_i \in U_i$, $y_i \in V_i$ and $P_1 g(x_1)\, x_1 \, g(x_2) \, x_2 \, \ldots g(x_j) x_j$ and $Q\, g(y_1) \, y_1 \, g(y_2) \, y_2 \, \ldots g(y_j)y_{j} $ are $M$-alternating vertex-disjoint paths in $G_m$. See Figure~\ref{fig:path} for an example of the paths we will create. 

\begin{figure}[htbp]
\centering
\includegraphics[width=0.5\textwidth]{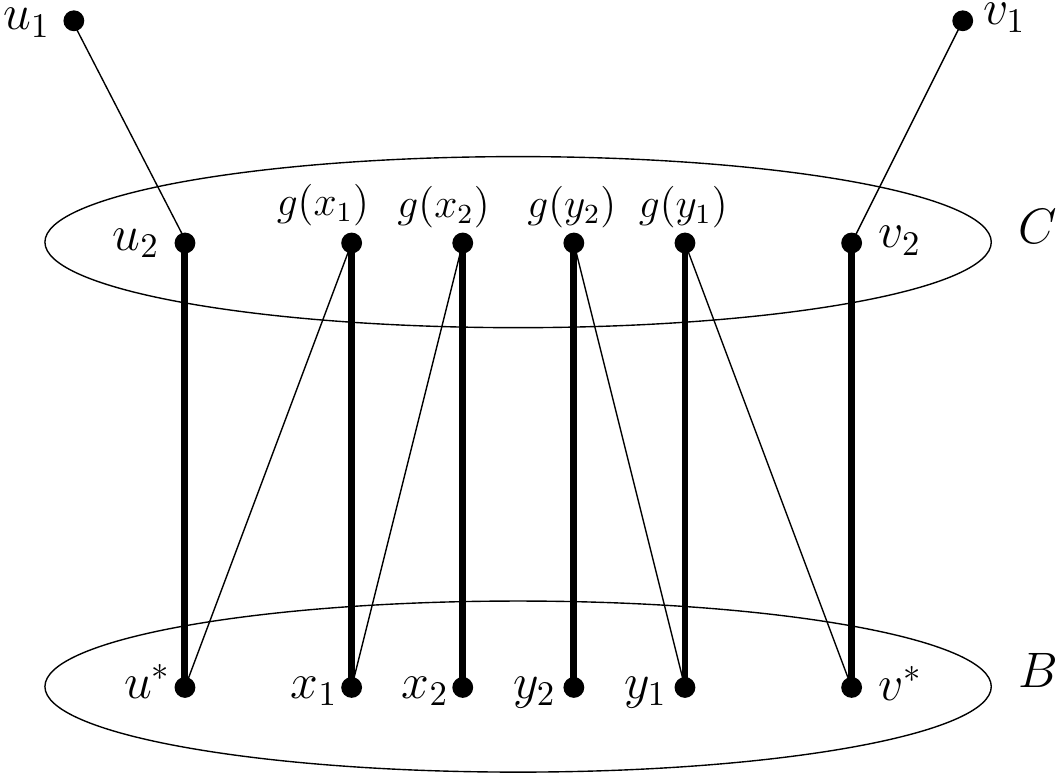}
\caption{An example of the way the $M$-alternating paths will be constructed. The edges of $M$ are bold.}\label{fig:path}
\end{figure}

In particular, we will construct these sets such that, for each $i$, we have $|U_i| \ge c \log n |U_{i-1}|$ and $|V_i| \ge c \log n|V_{i-1}|$. We will show that, given $U_i$ and $V_i$, we can create appropriate $U_{i+1}$ and $V_{i+1}$ until we reach some $k$ where $|U_k|, |V_k| \ge \frac{c}{4}n$. Such sets $U_k$, $V_k$, with the property described in the above paragraph, will give the family of paths required for the claim. 

These sets are constructed iteratively as follows. Set $U_0 := \{u^*\}$ and $V_0 := \{v^*\}$. Having defined $U_i$ and $V_i$, we will find disjoint sets $U_{i+1} \subseteq B$ and $V_{i+1} \subseteq B \setminus (P_1 \cup Q_1)$ such that:
\begin{enumerate}
\item $U_{i+1}$ and $V_{i+1}$ are disjoint from $\bigcup_{j=1}^{i}U_j \cup \bigcup_{j=1}^{i} V_j$;
\item $U_{i+1} \subseteq g(N(U_i))$ and $V_{i+1} \subseteq g(N(V_i))$.
\item Every vertex in $g(U_{i+1}) \cup g(V_{i+1})$ is large. 
\item If $|U_i| \leq \frac{n}{4 \log n} $, then $|U_{i+1}| = |V_{i+1}| \geq c \log n |U_i|$, \end{enumerate}

We will keep constructing until we reach $U_k$ and $V_k$ such that $|U_k| = |V_k|  \ge \frac{c}{4} n$. So the process runs for $O\left(\frac{\log n}{\log \log n}\right)$ steps. Observe that such a family of sets has the path property, and so constructing such a family will immediately give the claim.

Suppose that $U_i$ and $V_i$ have been defined for $0 \le i < k$. If $i=0$, set $X = N_{G_m}(u^*)$ and $Y = N_{G_m}(v^*)$. For $i > 0$, such that $|U_i| = |V_i| \leq \frac{n}{4 \log n}$, let $X:= N_{G_m}(U_i)$ and $Y:= N_{G_m}(V_i)$. For $i$ such that $|U_i| = |V_i| > \frac{n}{4 \log n}$, pick $U' \subseteq U_i$ and $V' \subseteq V_i$ each of cardinality $\frac{n}{4\log n}$ and set $X:= N_{G_m}(U')$ and $Y:= N_{G_m}(V')$. 

As $U_i \cup V_i \subseteq B$ and $B$ is an independent set, using (c) we may apply Lemma~\ref{lem:superexpansion} to each of $U_i$ and $V_i$. This gives that, for all $i$, we have $|X|, |Y| \ge 4c \log n |U_i|$.  As $U_i, V_i \subseteq B$, we have $X \cup Y \subseteq C$. Therefore we can choose disjoint sets $X'$ and $Y'$, each of cardinality $2c\log n|U_i|$ such that $X' \subseteq X$ and $Y' \subseteq Y$. Observe that $|P|,|Q| \le 5$ and 
$$\sum_{j=1}^{i} |U_j| \leq |U_i| \left(1 + (c \log n)^{-1} + (c \log n)^{-2} + \dots \right) = O(|U_i|). $$
Also note that, by Lemma~\ref{lem:agreeable} and construction, $g(X')$ and $g(Y')$ contain $O(|U_i|)$ small vertices. So it is possible to pick $U_{i+1} \subseteq g(X')$ and $V_{i+1} \subseteq g(Y')$ satisfying (a)--(d), as required. This completes the proof of the claim.
\end{proof}

As discussed above, the lemma follows directly from this claim. 
\end{proof}

Therefore,
whenever $|H_m| \ge 2$, the probability that $e_{m+1}$ extends a current maximum matching is bounded away from zero. We will use this fact to deduce that there cannot typically be a large number of steps where $|H_m| \ge 2$.

\begin{proof}[Proof of Lemma~\ref{lem:notwohelpers}]
Let $\eta$ be the constant from the statement of Lemma~\ref{lem:extendmatching} and $f(n):= \frac{n^{1-1/20}}{2}$. By our assumptions on the process, $\Gour{m_2}$ contains a matching of size $\frac{n}{2} - f(n)$. By Lemma~\ref{lem:extendmatching}, in a step $m$ with $|H_m| \ge 2$ the probability that the matching number increases is at least $\eta$. Note that from step $m_2$, the matching number cannot increase more than $f(n)$ times during the process. 

Let $X \sim \Bin(10f(n)/\eta, \eta)$. Then $ \mathbb{E}(X) = 10f(n)$ and, by the Chernoff bound (Theorem~\ref{Chernoff}) with $\epsilon=9/10$, we have
\begin{align*}
	\pr{X < f(n)} \leq e^{-\frac {81}{100} \cdot 10f(n)} = o(1).
\end{align*}
Thus with high probability, after $10f(n)/\eta$ steps with $|H_m| \ge 2$ a perfect matching is reached. As once a perfect matching is reached, there are no helpers, the result follows. 
\end{proof}

\subsection{States with zero or one helpers}\label{subsec:par} 

We now show that typically, in a constant proportion of time steps after $m_2$, we have that $|J_m|>0$ and $|J_m|$ is even, thus proving Lemma~\ref{lem:parity}. The lemma and its proof formalise the intuition that the number of steps where $|J_m|$ is even should be comparable to the number of steps where it is odd.
	
    We start the analysis from $m_2$ and, as before, we assume that $\Gour{m_2}$ satisfies the properties of Lemma~\ref{lem:precisedegrees},~\eqref{eq:num2} and Lemma~\ref{lem:agreeable}. We first prove the following lower bound on the number of isolated vertices in $\Gour{m_2}$. 
     \begin{lem} \label{claim:xm2}
     	With high probability, $\Gour{m_2}$ has at least $\frac 12 n^{1/4}$ isolated vertices.
    \end{lem}
    \begin{proof}
    	Let $p  = \frac {m_2}{ \binom n2}$ and let $Y_{m_2}$ be the number of isolated vertices in $G(n,p)$. One can check that $\er{Y_{m_2}} \sim ne^{-2m_2/n} \sim n^{1/4}$ and $\Var{Y_{m_2}} \sim \er{Y_{m_2}}$. So by Chebyshev's inequality, we have $\pr{Y_{m_2} \leq \frac 12 n^{1/4} } = o(1)$. Since this is a decreasing event, the same bound holds in $\Gall{m_2}$ by Lemma \ref{lem:gnpisfine}. Any isolated vertex in $\Gall{m_2}$ is also isolated in $\Gour{m_2}$, so the statement follows. 
    \end{proof}
    

\begin{proof}[Proof of Lemma~\ref{lem:parity}]

 %
    %

We modify the process $(G_m)$ to avoid a minor technical issue.
     The process $\left(\Gour{m}'\right)$ is defined identically to $(\Gour{m})$ with one change: for $m > 2n \log n$, the offered pair $e_m$ is drawn uniformly at random from the set $F := V^{(2)} \setminus\{ e_1, e_2, \dots, e_{2n \log n} \}$. The number of isolates and quasi-isolates in $\Gour{m}'$ is denoted by $J_m'$. 
   
   The purpose of the modification is so that in each step, the probability of being offered a particular edge is at most ${|F|}^{-1}$.
Note that, by definition, $\Gour{m}'$ coincides with $\Gour{m}$ in the first $2n \log n$ steps. Therefore, we may also assume that $\Gour{m}'$ satisfies the properties given by Lemma~\ref{lem:precisedegrees},~\eqref{eq:num2} and Lemma~\ref{lem:agreeable}. It suffices to prove the lemma for the modified process $\Gour{m}'$ as Theorem~\ref{thm:finalgraph}\eqref{itm:pm} implies that, with high probability, $\Gour{(2n \log n)}$ has a perfect matching and hence no quasi-isolates, and consequentially the same holds for $\Gour{2n \log n}'$. Hence we proceed with the proof for $\Gour{m}'$.

Let $\ell:=\lfloor \frac 14 n^{\frac 14} \rfloor$. For $k \in \left[2\ell\right]$, define 
$$T_k := |\left\{m \in \mathbb{N}: |J_m'| = k \right \}| \text{  and  } M(k): = \min \{m: |J_m'| \leq k \}.$$ As $J_{m_2}$ contains the isolates in $\Gour{m_2}$, Lemma~\ref{claim:xm2} implies that  $M\!\left(2\ell\right) \geq m_2$ with high probability. We remark that $\prm{M(k+1)}{\cdot}$ in the following claim is a slight abuse of notation because $M(k+1)$ is a random variable, but its definition is clear: it is the probability of an event conditioned on the pairs $e_1, e_2, \dots e_{M(k+1)}$.

For each $1 \leq k \leq 2 \ell$, we will define a new random variable and show that it stochastically dominates $T_k$. 
Let $X_k$ be mutually independent geometrically distributed random variables defined by
$$\pr{X_k = r} = \left(1 - \frac{4.1k}{n}\right)^{r}\frac{4.1k}{n} \text{\quad for } r \in \mathbb{N} \cup \{0\}.$$
Please note that this form of the geometric distribution is shifted by one compared to the more usual formulation. 


\begin{claim} \label{claim:Tdominated}
	For $ \log^2 n \leq k \leq 2\ell $ and $r \geq 0$,
	\begin{equation*}
		\prm{M(k+1)}{T_k  > r} \geq \pr{X_k > r}.
	\end{equation*}
\end{claim}
\begin{proof}
	We start by proving the statement for $r =0$, that is,
    \begin{equation}  \label{eq:tk0prob}
		\prm{M(k+1)}{T_k = 0} \leq \frac{4.1k}{n} = \pr{X_k = 0}.
	\end{equation}
	
	If $T_{k}=0$, then certainly $T_{k+1} > 0$ since we can lose at most two vertices of $J_m'$ at any step $m$.  Recall from above that $M(k+1)$ is the first time step such that $|J_{M(k+1)}'| = k+1$. 
	Let $U \subseteq V$ be the set of isolates, quasi-isolates and their partners in $\Gour{M(k+1)}'$, and note that $k+1 \leq |U| \leq 2k+2$. We will have $T_{k} = 0$ only if, from $M(k+1)$ onwards, an edge between two vertices of $U$ (which will always be acceptable) is offered before an acceptable edge between a vertex of $U$ and a vertex of degree at least two in $\Gour{M(k+1)}'$. We emphasise that this is only a necessary condition.

	In each step $m \geq M(k+1)$, the probability that an edge between two vertices of $U$ is offered is (using the definition of $\Gour{m}'$) at most $$\frac{\binom{2k+2}{2}}{\binom{n}{2} - 2 n \log n} = (1 + o(1))\frac{4k^2}{n^2}.$$ 

We have assumed (recall~\eqref{eq:num2}) that $\Gour{M(k+1)}'$ contains a subgraph $H$ which has a perfect matching and satisfies $|V(H)| \ge (1-o(1))n$ and $\delta(H) = \Omega(\log n)$. This implies that $\Gour{M(k+1)}'$ has at least $(1 - o(1))\frac{n}{2}$ vertices $v$ of degree at least 2 such that $uv$ is acceptable for any $u \in U$. So the probability that an acceptable edge between a vertex of $U$ and a vertex of degree at least 2 is offered is at least 
$$\frac{(1 - o(1))\frac{n}{2}\cdot(k+1) - m}{\binom{n}{2} - m} \ge (1 - o(1))\frac{k}{n},$$ 
as $\log^2 n \le k \le 2\ell$ and thus $m=o(nk)$. Therefore, for $\log^2 n \le k \le 2\ell$, we have
$\prm{M(k+1)}{T_{k}=0} \le (1 + o(1))\frac{4k}{n} \le \frac{4.1k}{n}$, as required.

Secondly, we show that for $r>0$
    \begin{equation} \label{eq:Tdominated2}
    	\prm{M(k+1)}{T_k >r \given T_k >0} \geq \left(1-\frac{4.1k}{n} \right)^r.
    \end{equation}
    In other words, we are bounding the probability that the number of isolates and quasi-isolates does not decrease in $r$ consecutive steps.
        	Assuming the event $\{T_k >0 \}$, consider the event that for all $m \in \{M(k)+1, M(k)+2, \dots, M(k)+r\}$, $e_m$ is not incident to an isolate, a quasi-isolate or their partner in $\Gour{m-1}$. Note that this event implies $T_k >r$, so we will derive a lower bound on its probability.
                
        For each $m \geq M(k)+1$, the number of vertex pairs incident to an isolate, quasi-isolate or its partner in $\Gour{m-1}$ is at most $2kn$. The probability that $e_m$ is one of those pairs is at most 
        $$\frac{2k n}{\binom{n}{2} - 2n \log n} \leq \frac{4 k }{n}\left( 1 + \frac{8 \log n}{n}\right).$$
Using this bound for each of the $r$ time steps, we get that
        \begin{equation*} \label{eq:rfailures}
        \prm{M(k+1)}{T_k >r \given T_k>0} \geq \left(1- \frac{4.1k}{n}\right)^r.
        \end{equation*}
Using~\eqref{eq:tk0prob}, we conclude that
        $$\prm{M(k+1)}{T_k > r}  =  \prm{M(k+1)}{T_k > r \given T_k >0} \prm{M(k+1)}{T_k >0} \geq  \left(1- \frac{4.1k}{n} \right)^{r+1}. $$
   
   On the other hand, a simple calculation yields 
   $\pr{X_k > r} =  \left(1- \frac{4.1k}{n} \right)^{r+1}$,
  so the statement follows.
\end{proof}

Recall that $\ell=\lfloor \frac 14 n^{\frac 14} \rfloor$ and let $T_{\text{even}}:= \sum_{i= \log^2 n}^{\ell} T_{2i}$. To prove Lemma~\ref{lem:parity}, we will show that $T_{\text{even}} = \Omega(n \log n)$ with high probability. It suffices to prove a lower bound on $X_{\text{even}}:= \sum_{i= \log^2 n}^{\ell} X_{2i}$.

\begin{claim}\label{claim:sumxisgood} For any $t\geq 0$,
$$\pr{T_{\mathrm{even}} > t} \ge \pr{X_{\mathrm{even}} > t}.$$
\end{claim}
\begin{proof}
	Since $T_{2\ell}, T_{{2\ell}-2}, \dots, T_{2k+2}$ are determined by the first $M(2k+1)$ steps of the process, Claim~\ref{claim:Tdominated} implies that for any $t \geq 0$,
    $$\pr{T_{2k} > t \given T_{2\ell}, T_{{2\ell}-2}, \dots, T_{2k+2}} \geq \pr{X_{2k}>t}. $$
    This implies the required statement (or in other words, that $T_{\text{even}}$ stochastically dominates $X_{\text{even}}$). For a proof, see, e.g.,~\cite[Lemma 21.22]{fk16}. We remark that in \cite{fk16} they consider random variables taking values in $[0, 1]$, but this assumption is never used in their proof.
\end{proof}

To conclude the proof of the lemma, we will show that $X_{\text{even}} \geq \frac{n \log n}{33} $ with high probability. The analysis is essentially that of the classical coupon collector problem, which  is described for instance in \cite{ferrante}. 

By definition $\er{X_{2i}} = \frac{n}{8.2i} - 1 $ and $\Var{X_{2i}} = (1 - \frac{8.2i}{n})\left(\frac{n}{8.2i}\right)^2$. Also recall that $\ell=\lfloor \frac 14 n^{\frac 14} \rfloor$. By linearity of expectation we have
$$\er{X_{\text{even}}} = \sum_{i = \log^2 n}^{\ell} \left(\frac{n}{8.2i} - 1\right)
= \frac{n}{8.2}\left(\frac{1}{5 \log n} + \cdots + \frac{1}{\ell}\right) - O(\ell) \sim \frac{n}{8.2}\cdot \log \ell > \frac{n \log n}{32.9} .$$ 

We now calculate the variance of $X_{\text{even}}$ in order to use Chebyshev's inequality to show that $X_{\text{even}}$ does not deviate too far from its expectation. As the variables $X_{2i}$ are independent, we have
$$\Var{X_{\text{even}}} = \sum_{i=\log^2 n}^{\ell}\Var{X_{2i} }= \sum_{i=\log^2 n}^{\ell} O\left(\frac{n^2}{i^2} \right)  \le  O\left(n^2 \right) \sum_{i=1}^{\infty} \frac{1}{i^2} = O\left(n^2 \right).$$
So by Chebyshev's inequality, we have 
$$\pr{|X_{\text{even}} - \mathbb{E}(X_{\text{even}})| \ge \sqrt{n\log n} } = O\left(\frac{1}{\log n}\right) = o(1).$$
It follows that, with high probability, $X_{\text{even}} \geq \frac{n \log n}{33} $.

So, by Claim~\ref{claim:sumxisgood}, with high probability $T_{\text{even}} \geq \frac{n \log n}{33} $. Assuming the statement of Lemma~\ref{claim:xm2}, all the time steps counted in $T_{\mathrm{even}}$ occur after $m_2$, as required.
\end{proof}

We remark that we are not able to give an optimal constant for Lemma~\ref{lem:parity} because we do not establish full control of the transition probabilities (i.e.,~$X_k$ is just a convenient lower bound for $T_k$), and we are ignoring any delay that arises before the time step $m_2$. Therefore, our computations are fairly crude.

\subsection{Decay of the number of isolated vertices} \label{subsec:mart}

We complete the proof of Theorem~\ref{thm:delayedpm}  by tracking the number of isolated vertices (not quasi-isolates). Our goal is to show that, with high probability, at time $m= (\frac{1}{2} + \frac{1}{70})n \log n$, the graph $G_m$ contains isolated vertices. Denote the number of isolated vertices in $G_m$ by $I_m$. From the previous section, we know that $\Gour{m}$ contains no helpers in a constant proportion of steps and we wish to show that this `slows down' the decay of the number of isolates compared to $\Gall{m}$.

Before presenting the techinical details of the proof, let us explain heuristically what we will prove. As before, throughout the section we continue to condition on the events given by Lemmas~\ref{lem:agreeable},~\ref{claim:xm2} and \eqref{eq:num2}. Fix $\phie = \phie(n) = (\log n)^{-1/10}$. Say that a time step $m$ is \emph{unhelpful} if $|H_m| = 0$ and is not $\phie$-flexible, that is, $|D_m| < \left(\frac 12 + \phie \right)n$ (recalling that $D_m$ is the set of vertices contained in some optimal cover at time step $m$). 
The expected change in the number of isolates at a time step depends on whether the step is unhelpful or not. Informally, we have that `on average'
\begin{equation}\label{eq:Idef}
I_i - I_{i+1} \le I_i Z_i,
\end{equation}
where
\begin{equation}\label{eq:zidef}
        Z_i = 
    \begin{cases}
    	\frac{1+4 \phie}{n}, &  \text{if } i \text{ is unhelpful,}  \\
        \frac{2+\phie}{n}, & \text{otherwise}.
    \end{cases}
\end{equation}
The formal statements and justification follow in the proof.
Solving \eqref{eq:Idef}, we might expect that
\begin{equation}
\label{eq:eqsol}
I_m\ge I_0\exp\left(-\sum_{i=1}^m Z_i \right),
\end{equation}
and we will use martingales and Freedman's Inequality (Theorem~\ref{thm:freedman}) to track the isolates and show that they behave essentially as in \eqref{eq:eqsol}. 
The reader may find it helpful to compare this decay rate with $\Gnm$, where `on average', the number of isolated vertices decreases by a factor of $\left(1-\frac 2n (1+o(1)) \right)$ per time step. Martingale concentration inequalities can be used to show that with high probability, the number of isolated vertices in the graph $\Gnm$ is asymptotically $ne^{-2m/n}$ when $m \leq \frac{n \log n}{2}$.\\

There are two technicalities we would like highlight. Firstly, define 
\begin{equation}
\label{rdef}
r:= \min \{i: I_i \leq \lfloor\log n  \rfloor\}.
\end{equation} 
We will work with the truncated variable $\hat{I}_m = \max\{I_m, I_r \}$ as the difference equations stop holding if $I_m$ is too small. Secondly, we apply Freedman's inequality (as apposed to, e.g.~Azuma's) in order to work with larger differences as long as they only occur rarely.

Let us now proceed with the details of the proof.

     \begin{proof}[Proof of Theorem \ref{thm:delayedpm}]
As before, let $\phie := (\log n)^{-1/10}$. Let 	
	$$A_m:= |\{m_2 \le i \le m_2 +m -1 : \text{ step } i \text{ is unhelpful}\}|.$$

The following claim says that the number of isolated vertices in our process decays slower than in the corresponding $\Gall{m}$.
    \begin{claim} \label{lem:trackisolates}
	Let $m_2 \leq m \leq 2n\log n$. With high probability, 

    	\begin{equation}		\label{eq:trackisolates}
			\log\left( \frac{\ck_{m_2+m}}{\ck_{m_2}} \right) \geq  -\frac {2m}{n}  + \frac{A_m}{n} - \frac{6\phie m}{n}.
		\end{equation}
	\end{claim}
	\begin{proof}
	For $  m_2 \leq i \leq m_2 + m$, define random variables $(Y_i)$  by $Y_{i+1}  := \log \left(\frac{\ck_{i+1}}{\ck_{i}} \right) + Z_{i}$, where $Z_{i}$ is defined in \eqref{eq:zidef}. By definition,
		$$\sum_{i = m_2 }^{m_2 + m-1} Z_i = \frac{1+4\phie}{n}\cdot A_m + \frac{(2+\phie)}{n}\cdot(m-A_m).$$
The sequence $(Y_i)$ is set up in order to track the logarithm of $\ck_i$, since we expect $\ck_i$ to decrease exponentially. Moreover, the definition of $Z_i $ has been chosen to ensure that
$\erm{i}{Y_{i+1}} \geq 0$, i.e., 
 $\sum_{i=m_2}^{m} Y_i$ is a submartingale. To prove this, we first estimate the probability that $\ck_{i+1} = \ck_i-1$.
    
    First consider the case where $i < r$, where $r$ is defined in (\ref{rdef}), and $Z_i = \frac{1 + 4\phie}{n}$, so $i$ is unhelpful. As $i < r$, $\ck_i = I_i$. In this case, a pair $uv$ with $v$ isolated is acceptable for $\Gour{i}$ only if $u \in D_i \cup J_i$, or if $u$ is the partner of a quasi-isolate. Since $|D_i|< \left(\frac 12 + \phie \right)n$ and $|J_{m_2}|< n^{0.95}$ by \eqref{eq:num2}, we have
    $$\prm{i}{I_{i+1} = I_i-1}  \leq \frac{I_i(|D_i|+2|J_i|)}{\binom{n}{2}-i} \leq \frac {(1+3\phie)I_i}{n}.$$
    
	Now we compute $\erm{i}{Y_{i+1}}$. We have
\begin{align*} 
	\erm{i}{Y_{i+1}}  &= \prm{i}{I_{i+1} = I_i-2} \log \left(1- \frac{2}{I_i} \right) +\prm{i}{I_{i+1}=I_i-1}  \log \left(1 - \frac{1}{I_i} \right)+Z_i \\ & 
		\geq  \frac{I_i^2}{2}\cdot \frac{2}{n(n-5 \log n)}\cdot \log \left(1 - \frac{2}{I_i} \right)+\frac{(1+3\phie)I_i}{n}\log \left(1- \frac{1}{I_i} \right) + Z_i.
	\end{align*}
The first summand, corresponding to $\{I_{i+1} = I_i-2 \}$, is of a lower order. This event happens when we offered one of ${|I_i| \choose 2}$ acceptable pairs with both vertices in $I_i$.
Its denominator is $n - 5 \log n$ to account for $i \leq 2 n \log n$ edges which have been offered so far. We also use the inequality $\log(1-x) \geq - x - x^2$ for $x < 0.1$ to obtain the bound
    \begin{align*}
    \erm{i}{Y_{i+1}} 
    & \geq O \left(\frac{I_i}{n^2} \right) + \frac{(1+3\phie)I_i}{n} \left(-\frac{1}{I_i}- \frac{1}{I_i^2} \right) + \frac{1+4\phie}{n}
   	\geq 0.
    \end{align*}

    In the case $i < r$ and $Z_i = \frac{2+\phie}{n}$, we do a similar computation to see that $\erm{i}{Y_{i+1}}\geq 0$. The only change is in the second summand, since now any pair containing an isolated vertex $v$ can be acceptable and therefore the probability that $v$ is not isolated in $\Gour{i+1}$ is at most $\frac{2}{n- 5 \log n}$.   This gives
	\begin{align*}
    	\erm{i}{Y_{i+1}} 
			&\geq  \frac{I_i^2}{2}\cdot \frac{2}{n(n- 5 \log n)} \log \left(1 - \frac{2}{I_i} \right)+\frac{2I_i}{n-5 \log n}\log \left(1- \frac{1}{I_i} \right) + Z_i \\
		&\geq   O \left(\frac{I_i}{n^2} \right) + \frac{2I_i}{n-5 \log n} \left(-\frac{1}{I_i}- \frac{1}{I_i^2} \right) + \frac{2+\phie}{n} \geq 0.
	\end{align*}
    In the case $i \geq r$, we have $\ck_{i+1} = \ck_i = \ck_r$  by \eqref{rdef}, so $\erm{i}{Y_{i+1}} = 0 + Z_i >0$. 
    
    To control $\sum_{i=m_2}^{m}Y_i$ we will use Freedman's Inequality (Theorem~\ref{thm:freedman}), so let us bound the predictable quadratic variation of the process. Since $\log \left( \frac{ \ck_{i+1} }{ \ck_i } \right) \leq 0$, 
	 \begin{align*}
		\erm{i}{Y_{i+1}^2} &\leq \erm{i}{  \log^2 \left( \frac{\ck_{i+1}}{\ck_i}\right) } + \erm{i}{Z_i^2} \\
			& \leq \prm{i}{\ck_{i+1} = \ck_i-2} \log^2 \left(1- \frac{2}{\ck_i} \right) +\prm{i}{\ck_{i+1}=\ck_i-1}  \log^2 \left(1 - \frac{1}{\ck_i} \right)+Z_i^2.
		\end{align*}
The first two summands in the following inequality correspond to the event where $\ck_{i+1}<\ck_i $, which only occurs for $i > r$. The third term corresponds to the event where $\ck_{i+1}=\ck_i$. 
	We use the notation $\log^2 x = (\log x)^2$.  Once again, we use the inequality $-\log(1-x) \leq 2x$ (for $x < 0.1$) to conclude that
    \begin{align*}
     \erm{i}{Y_{i+1}^2} 
     & \leq \frac{2\ck_i^2}{n^2}\cdot \frac{16}{\ck_i^2} + \frac{4\ck_i}{n}\cdot \frac{4}{\ck_i^2} + \frac{5}{n^2} 
     \leq \frac{20}{n \ck_i} .
    \end{align*}
   Summing over $i$ and using $\frac{1}{\ck_i} \leq \frac{1}{\log n}$, we get 
   $$W_m = \sum_{i = m_2}^{m_2 + m} \erm{i}{Y_{i+1}^2} \leq \frac{20m}{n \log n} \leq 40.$$
   Applying Theorem \ref{thm:freedman} with $\ell = \sqrt{\log n}$, we get that with high probability
   $\sum_{i = m_2+1}^{m_2 + m} Y_{i} \geq - \sqrt{\log n}$. Therefore
\begin{eqnarray*}
\log\left( \frac{\ck_{m_2+m}}{\ck_{m_2}} \right) &=& \sum_{i = m_2+1}^{m_2 + m} (Y_{i} -Z_{i-1}) \geq - \sqrt{\log n}-\sum_{i = m_2}^{m_2 + m-1} Z_i\\
& = &  - \sqrt{\log n} - \frac {2+\phie}{n} \cdot(m-A_m)  - \frac {1+4\phie}{n} \cdot A_m  \geq -\frac{2m}{n}+ \frac{A_m}{n} - \frac{6 \phie m}{n}.
\end{eqnarray*}   
   
\end{proof}

		 Let $C := \frac{1}{34}$ and $m := \left( \frac 18 + \frac{34C}{70} \right)n \log n$. Assume that $|H_i| = 0$ and $|J_i |>0$ for at least $C n \log n$ steps $i \geq m_2$. By Lemma~\ref{lem:zerohelpers}, this holds with high probability. If one of those steps satisfies $i \geq m +m_2$, then $|J_{m_2+m}|>0$.  Lemma~\ref{lem:agreeable} with $a = 2$ implies that with high probability there are no quasi-isolates at time $m_2 + m$, so $I_{m_2+m}>0$ as required. Therefore we may assume that all the steps $i$ with $|J_i|>0$ occur before $m_2+m$. This, along with the fact that (by Lemma~\ref{lem:coverfrozen}) the number of $\phie$-flexible steps is typically $o(n \log n)$, implies that $A_m \geq C n \log n(1-o(1))$.
        By substituting in our value of $m$, we have 
$$\frac{2m}{n} - \frac{A_m}{n} + \frac{6 \phie m}{n}<  \left ( \frac14 + \frac{34 C}{35}- C +o(1)\right) \log n = \left(\frac{1}{4}-\frac{C}{35} +o(1) \right)\log n.$$
		This along with  Claim \ref{lem:trackisolates} implies that with high probability,
		$$\log \left(\frac{\ck_{m_2+m}}{\ck_{m_2}}\right) \geq -\frac{2m}{n} + \frac{A_m}{n} - \frac{6 \phie m}{n} > - \left ( \frac14 - \frac{C}{40} \right) \log n.$$
	Since $\ck_{m_2} = I_{m_2} > \frac 12 n^{1/4}$ by Lemma~\ref{claim:xm2}, we get that $\ck_{m_2 + m} \geq \frac 14 n^{\frac 14 - \frac 14 + \frac{C}{40}} >n^{C/50} > \lfloor \log n \rfloor$, and therefore $I_{m_2+m}> 0$. It remains to verify that the constant is as stated in Theorem \ref{thm:delayedpm}, that is, $m_2 + m = \left(\frac 38 + \frac 18 + \frac{34C}{70} \right)n \log n =  \left(\frac 12 + \frac{1}{70} \right) n \log n$. \\

	\end{proof}

	\section{Concluding Remarks and Open Problems} \label{section:open}
        In this paper we studied the random graph process which at every step maintains the property that 
        the matching and cover number of the current graph coincide. We have made progress in understanding the properties of the graph obtained at various stages of this process. 
        Yet several interesting problems remain open. Theorems~\ref{thm:finalgraph} and \ref{thm:delayedpm} give us bounds on the typical appearance of a perfect matching in $\Gour{m}$. We have a heuristic argument which suggests the threshold is $(1+o(1))\frac{3}{4}n\log n$ and we wonder if this can be made rigorous. 
        
        \begin{ques}
        What is the  threshold for appearance of a perfect matching in $\Gour{m}$?
        \end{ques}
        
        Lemma~\ref{lem:coverfrozen} says that with high probability, most of the time when $m \geq c n \log n$ for a small constant $c>0$, the union of optimal covers $D_m$ has $\frac n2 (1+o(1))$ vertices. We wonder if this holds \emph{all the time}.
        
        \begin{ques}
        Is it true that, with high probability, the union of optimal covers $D_m$ have $\frac n2 (1+o(1))$ vertices \emph{for all} $m \geq c n \log n$ for a small positive constant $c$?
        \end{ques}

        In light of Theorem~\ref{thm:finalgraph}, it would also be interesting to determine precisely when $\Gour{m}$ has a unique optimal vertex cover and to determine the correct order of the number of vertex pairs that are not present in $G_N$ (in particular, whether it is sublinear in $n$). In order to get precise bounds for any of the properties discussed above, we believe it is necessary to analyse more carefully the earlier stages of the process (before $m = \frac{3}{8}n\log n$). We have made essentially no attempt to do this. 
     
        Another direction of research could be to study a corresponding random K\H{o}nig hypergraph process.
        Here we fix $r \ge 3$ and $n$ such that $r|n$. Let $N:= \binom{n}{r}$. $e_1,\ldots,e_N$ be a uniformly random ordering of the edges of the complete $r$-uniform hypergraph on $n$ vertices. Let $G_0$ be an independent set of $n$ vertices. 
        For $i = 0,\ldots,N-1$ we define $G_{i+1} :=  G_i + e_{i}$ if $\tau(G_i + e_i) = \nu(G_i + e_i)$, and $    G_{i+1}: =G_i$ otherwise.
        As in the graph process, we expect to reach a hypergraph $G$, where $\tau(G) = \nu(G) = \frac nr$. We also expect $V(G) = T \cup I$, where $|T| = \frac nr$, nearly all edges touching $T$ are present and $I$ is independent. 
        It seems that the hypergraph process will require a different set of tools, as there is no concept analogous to alternating paths.    
        
        Finally, it is also interesting to study random graph processes preserving other global properties, for example the property of being a perfect graph.

        \vspace*{5pt}

\vspace*{5pt}

\noindent \textbf{Acknowledgement.} Part of this research was done when the second and third authors visited ETH Z\"{u}rich and the fourth author
visited Tel Aviv University. We want to thank both institutions for their hospitality. The third author would also like to thank the
London Mathematical Society for making this visit possible. The first author would like to thank Matthew Kwan and Vincent Tassion for helpful discussions.\\
We are grateful to the anonymous referee for their valuable comments, which improved the presentation of the paper.

\bibliographystyle{abbrv} 
\bibliography{koenig_process.bib}

\end{document}